\documentclass[journal]{IEEEtran}

\usepackage[utf8]{inputenc}
\usepackage{amsmath}
\usepackage{amssymb}
\usepackage{color}
\usepackage{dsfont}
\usepackage{graphicx}
\usepackage{mathdots}
\usepackage{multirow}
\usepackage{siunitx}
\usepackage{booktabs} 
\usepackage{rotating}
\usepackage{caption}
\usepackage[colorlinks=true, allcolors=blue]{hyperref}
\usepackage{cleveref}
\usepackage{theorem}
\usepackage{algorithm}
\usepackage{algpseudocodex}
\usepackage{mathtools}
\usepackage{tabularx}
\usepackage{subcaption} 

\def\sbmm#1{\ensuremath{\boldsymbol{#1}}}          
\def\sdm#1{\ensuremath{\mathrm{#1}}}               
\def\sbv#1{\ensuremath{\mathbf{#1}}}               

\def\x{\sbv{x}}
\def\s{\sbv{s}}
\def\y{\sbv{y}}
\def\u{\sbv{u}}

\def\bx{\sbmm{x}}

\def\by{\sbmm{y}}
\def\bu{\sbmm{u}}
\def\bw{\sbmm{w}}
\def\bc{\sbmm{c}}

\def\bY{\sbmm{Y}}
\def\bA{\sbmm{A}}
\def\bU{\sbmm{U}}
\def\bW{\sbmm{W}}
\def\bC{\sbmm{C}}
\def\bS{\sbmm{S}}
\def\bX{\sbmm{X}}
\def\bI{\sbmm{I}}
\def\bQ{\sbmm{Q}}

\def\C{\mathcal{C}}
\def\W{\mathcal{W}}
\def\X{\mathcal{X}}
\def\S{\mathcal{S}}
\def\A{\mathcal{A}}

\def\Rx{\widehat{\sbmm{R}}_{\x}}

\def\N{\mathbb{N}}
\def\R{\mathbb{R}}
\def\nfirst{\{1,\ldots,N\}}
\def\kfirst{\{1,\ldots,K\}}

\def\dom{\operatorname{dom}}

\def\Diag{\operatorname{Diag}}
\def\diag{\operatorname{diag}}
\def\tr{\operatorname{tr}}
\def\argmin{\operatorname{argmin}}
\def\min{\operatorname{min}}
\def\max{\operatorname{max}}
\def\prox{\operatorname{prox}}

\def\Jivareg{J_{\rm IVA-G}^{\text{Reg}}}

\newtheorem{theorem}{Theorem} 
 
\newtheorem{remark}{Remark}
\newtheorem{assumption}{Assumption} 
 
\newtheorem{lemma}{Lemma}

\begin{document}

\title{An Effective Iterative Solution for Independent Vector Analysis with Convergence Guarantees}
\author{Cl\'ement Cosserat,~\IEEEmembership{Student Member}~IEEE, Ben Gabrielson,~\IEEEmembership{Student Member},~IEEE, Emilie Chouzenoux, \IEEEmembership{Senior Member},~IEEE, Jean-Christophe Pesquet, \IEEEmembership{Fellow},~IEEE, and T\"ulay Adali, \IEEEmembership{Fellow}~IEEE}
\maketitle
  \thanks{}

\markboth{IEEE TRANSACTIONS ON SIGNAL PROCESSING, VOL.~XX, NO.~XX
}{Cosserat \MakeLowercase{\textit{et al.}}: Title}

\begin{abstract}
Independent vector analysis (IVA) is an attractive solution to address the problem of joint blind source separation (JBSS), that is, the simultaneous extraction of latent sources from several datasets implicitly sharing some information. Among IVA approaches, we focus here on the celebrated IVA-G model, that describes observed data through the mixing of independent Gaussian source vectors across the datasets. IVA-G algorithms usually seek the values of demixing matrices that maximize the joint likelihood of the datasets, estimating the sources using these demixing matrices. Instead, we write the likelihood of the data with respect to both the demixing matrices and the precision matrices of the source estimate. This allows us to formulate a cost function whose mathematical properties enable the use of a proximal alternating algorithm based on closed form operators with provable convergence to a critical point. After establishing the convergence properties of the new algorithm, we illustrate its desirable performance in separating sources with covariance structures that represent varying degrees of difficulty for JBSS.
\end{abstract}

\begin{IEEEkeywords}
IVA, $\textbf{PALM}$ Algorithm, Maximum-Likelihood, Blind Source Separation, Proximal methods
\end{IEEEkeywords}

%

\section{Introduction}
\fontsize{10pt}{11pt}\selectfont

\IEEEPARstart{T}{he} \emph{blind source separation} problem (BSS) aims at factorizing a data matrix as a product of a mixing matrix and a source data matrix.
BSS is thus a data-driven method to extract latent features from a dataset, which have a broad variety of uses. It offers a wide range of applications in signal processing and engineering including neuroimaging data analysis \cite{calhoun_unmixing_2006}, communications \cite{BSSwireless,contrastsforBSS}, remote sensing \cite{nielsen_multiset_2002}, to name a few  \cite{Jutten_Comon_handbook}. The latent sources are interpreted as physical quantities of interest that cannot be measured directly, e.g. brain connectivity networks \cite{calhoun_unmixing_2006,Adali2014,calhoun_multisubject_2012}, or can be used as features for further tasks such as classification \cite{electro-coulography}. The BSS problem generalizes to the \emph{joint blind source separation} problem (JBSS) when multiple datasets are analyzed jointly to benefit from their shared information. JBSS is necessary to fully analyze datasets that share similarities. This case presents itself when the datasets contain measures of a same phenomenon for different subjects \cite{Adali2014,gabrielson_joint-iva_2020}, different measurement methods \cite{multimodal_data_fusion}, or more generally for various modalities \cite{boukouvalas_independent_2021}. Allowing the interaction between the datasets leads JBSS to achieve more accurate separation than multiple separate BSS in general \cite{Adali2014,Anderson2012}.

One way to address BSS is to model the rows of the source dataset as samples of mutually independent random variables or random process called sources and the datasets as mixtures of independent sources \cite{IndependenceMeasure}. This approach is known as \emph{independent component analysis} (ICA) \cite{comon1994independent}, which has been one of the most popular ways to achieve ICA due to its uniqueness guarantees under very general conditions. ICA can be generalized to \emph{independent vector analysis} (IVA) \cite{Adali2014,Anderson2012,Anderson_IVA_IC_PB,adali_ica_2019,KimIVA} to address the JBSS problem. In this case, sources are considered as random vectors (or random process vectors), called \emph{source component vectors} (SCVs). Each entry of a SCV accounts for one source in a given dataset. In IVA, now the SCVs are assumed to be independent rather than univariate sources as in ICA. Within each SCV, the statistical dependence (correlation for IVA-G) across the datasets are taken into account. IVA methods gather a family of algorithms that aim at recovering the SCVs and mixing coefficients that generated the observed data. The JBSS problem can be addressed with many other methods such as groupICA \cite{calhoun_multisubject_2012,adali_ica_2019}, or joint ICA \cite{jointICA_original}, but IVA is more powerful as it offers a greater flexibility and helps preserve the individual variability represented by each dataset, e.g., in multisubject analyses \cite{subject_variability,laney2015capturing}. Moreover, it enables a common factorization of the datasets without needing to realign the sources $a ~ posteriori$, which can be costly, thus alleviating the inherent permutation ambiguity across the datasets. 

IVA is typically formulated using a \emph{maximum of likelihood} (ML) estimator, which means that the estimation is done by solving an optimization problem where the cost-function is derived from the log-likelihood of the data. In \cite{Anderson2012}, we see IVA-G presented as a minimization of the mutual information of the SCVs to maximize their independence, those two formulations are equivalent when the number of samples tends to infinity \cite{Adali2014}.

In this paper, we focus on the case of non-degenerate, centered Gaussian SCVs, placing ourselves in the so-called IVA-G framework, as presented in \cite{Anderson2012}. This model is convenient, as the SCVs are fully described by their covariance matrix (or equivalently by their precision matrix). Besides, since IVA-G algorithms only make use of \emph{second order statistics} (SOS) of the data, this makes them the simplest and most efficient among the IVA algorithms. In the Gaussian case, having the demixing matrices gives a natural estimate of the covariance matrices of the SCVs through their empirical covariance, this is why the cost function in \cite{Anderson2012} only depends on the demixing matrices. In this work however, we choose to write the ML such that it explicitly takes the SCVs \emph{probability density function} (pdf) as an input variable, to benefit from analytical properties of the resulting cost.

So far, the IVA-G algorithms are based on standard optimization methods, like Newton's method or gradient descent, without demonstration of explicit convergence guarantees. In this article, we show that our proposed cost function to jointly search for the Gaussian SCVs and demixing matrices generalizes the cost function in \cite{Anderson2012}, in the sense that replacing the precision matrices of the SCVs by the inverse of their empirical covariance matrices for the given demixing matrices, we recover a cost function that only depends on those demixing matrices and that is equal to the one in \cite{Anderson2012} up to a constant. The choice of a cost function that jointly acts on the sought demixing and covariance matrices offers an attractive structure enabling the explicit incorporation of priors and constraints, and the use of mathematically sound minimization algorithms. Hence, we design a \emph{Proximal Alternating Linearized Minimization} algorithm ($\textbf{PALM}$), dedicated to IVA-G problem, and show that it converges to a critical point under some mild assumptions. 
The resulting scheme, denoted $\textbf{PALM-IVA-G}$, is also shown to be fast, and to establish a reliable estimation performance in practice, and at least as good as the state-of-the-art methods for IVA-G.

Our contributions can be summarized as follows. First, we provide a novel variational formulation for IVA-G, introducing a new cost function and establishing its mathematical properties. Second, building upon these properties, we design the $\textbf{PALM-IVA-G}$ algorithm, to solve the resulting minimization problem, and show its convergence under mild assumptions. Third, we numerically illustrate the desirable performance of our approach, in comparison with state-of-the-art approaches, on a number of scenarios that cover various degrees of dependence across the datasets.

The paper is organized as follows. Section \ref{sec:IVA-G} introduces the JBSS problem, the IVA-G framework to address it, and  presents our cost function based on maximum likelihood estimation. Section \ref{sec:Palm} then focuses on our proposition of an original iterative algorithm to solve the IVA-G problem. Convergence results are presented in Section \ref{sec:conv}. Numerical experiments assessing the validity and good performance of our approach are presented in Section \ref{sec:expe}. Section \ref{sec:conc} concludes the paper.

\section{The IVA-G problem}
\label{sec:IVA-G}
\subsection{Notation}
Throughout the paper, we use bold upper case symbols for matrices, bold lower case symbols for column or row vectors, calligraphic upper case symbol for tensors of order three or more, and regular lower case symbols for scalars. Italic font is used for deterministic quantities while regular one is used for random quantities. For instance, let $\u = [\sdm{u}_1,\ldots,\sdm{u}_N]^\top$ be an $\R^N$-valued random vector, for which we draw $V$ realisations that we stack, columnwise, in a matrix $\bU \in \R^{N \times V}$. The latter can also be written as $\bU = [\bu_1,\ldots,\bu_N]^\top$ with, for every $n \in \nfirst$, $\bu_n \in \R^V$ a column vector of $V$ realizations of the scalar random variable $\sdm{u}_n$.\\ 

The subset of non-singular (resp. symmetric) matrices of $\R^{N \times N}$ is denoted $\mathcal{GL}_N(\R)$ (resp. $\mathcal{S}_N$). $\sbmm{I}_N$ is the $N\times N$ identity matrix. We also denote $\mathcal{S}_N^+$ the set of positive semi-definite matrices of size $N$, and $\mathcal{S}_N^{++} = \mathcal{S}_N^+ \bigcap \mathcal{GL}_N(\R)$ the set of \emph{positive definite} (PD) matrices. $\|\cdot\|$ denotes the Frobenius norm for any vector, matrix, and tensor of order three or more.\\

For any matrix $\sbmm{A}$, we denote by $\| \sbmm{A}\|_{\rm S}$ its spectral norm, equal to the largest singular value of $\sbmm{A}$. For every $n \in \{1,\ldots,N\}$, $\sbmm{a}_n^\top$ denotes the $n$-th row of $\sbmm{A}$. If $\sbmm{A}$ is a square matrix, we note $\boldsymbol{\sigma}_{\sbmm{A}} = (\sigma_{\sbmm{A},l})_{1 \leq l \leq N} \in \R^N$ the vector of its singular values, $\tr(\sbmm{A})$ its trace, and $\det(\sbmm{A})$ its determinant. $\diag(\sbmm{A}) \in \R^N$ is the vector whose entries are the diagonal coefficients of $\sbmm{A}$, whereas for $\sbmm{a} \in \R^N$, $\Diag(\sbmm{a})$ (with an upper case) is the diagonal matrix whose coefficients are the components of $\sbmm{a}$. We also note $\Diag(\sbmm{A})$ the diagonal matrix whose diagonal coefficients are the same as those of $\sbmm{A}$. Finally, for any $(N,M)\in (\mathbb{N}\setminus\{0\})^2$, we consider the matricial scalar product defined as $(\forall (\sbmm{A}, \sbmm{B}) \in (\R^{N \times M})^2) ~ \langle \sbmm{A} \mid \sbmm{B} \rangle = \tr(\sbmm{A}^\top \sbmm{B})$.


\subsection{JBSS problem}

Let $K,N,V$ be positive integers, we consider $K$ datasets denoted $(\bX^{[k]})_{1 \leq k \leq K}$ where $\forall k \in \kfirst, \bX^{[k]} \in \R^{N \times V}$ is a real-valued matrix. For instance, in an fMRI analysis problem, for $n \in \{1,\ldots,M\}$, and $k \in \kfirst$, the $n$-th row of $\bX^{[k]}$, $(\bx_n^{[k]})^\top$ could model the blood-oxygen-level-dependent (BOLD) contrast in the $V$ voxels of a 3D model of the brain, measured at acquisition time $n$ for the $k$-th subject, within a cohort of $K$ patients \cite{calhoun_unmixing_2006,calhoun_multisubject_2012}.
We assume that the observed datasets are obtained by a linear mixing of latent source datasets, i.e, $(\forall k \in \kfirst)$:

\begin{equation}
\label{matricial mixing equations}
     \bX^{[k]} = \bA^{[k]}\bS^{[k]} \in \R^{N \times V},
\end{equation}
where $\bA^{[k]} = \big( a^{[k]}_{m,n} \big)_{1 \leq m \leq N, 1 \leq n \leq N} \in \R^{N \times N}$ is a square non-singular mixing matrix and $\bS^{[k]} \in \R^{N \times V}$ is a latent matrix whose coefficients are typically interpreted as hidden features of the phenomenon described by $\bX^{[k]}$. The JBSS problem consists in estimating simultaneously for all $k \in \kfirst$, both $\bA^{[k]}$ and $\bS^{[k]}$ from $\bX^{[k]}$. We estimate the inverse of the mixing matrices by the so-called demixing matrices $(\bW^{[k]})_{1 \leq k \leq K}$, and deduce estimates $(\bY^{[k]})_{1 \leq k \leq K}$ for the source datasets $(\bS^{[k]})_{1 \leq k \leq K}$ by calculating
\begin{equation}
\label{build source estimates}
(\forall k \in \kfirst) \quad \bY^{[k]} = \bW^{[k]}\bX^{[k]}.
\end{equation} 

In a nutshell, using tensor notations, JBSS amounts to providing an
estimate $\mathcal{Y} = [\bY^{[1]},\ldots,\bY^{[K]}]$ of the source tensor $\S = [\bS^{[1]},\ldots,\bS^{[K]}]$ via the demixing tensor $\W = [\bW^{[1]},\ldots,\bW^{[K]}]$, which is an estimate of the slice-wise inverse of the mixing tensor $\mathcal{A} = [\bA^{[1]},\ldots,\bA^{[K]}]$, given the datasets $\mathcal{X} = [\bX^{[1]},\ldots,\bX^{[K]}]$. Here, $\S,\X$ and $\mathcal{Y} \in \R^{N \times V \times K}$ and $\A$ and $\W \in \R^{N \times N \times K}$.

\subsection{IVA model}

Assuming that the datasets share some information which can be leveraged to separate the sources more accurately than if dealt with individually, IVA \cite{KimIVA} models this interdependence of the $\bX^{[k]}$ through statistical links between the latent datasets. More precisely, the rows of the latent datasets with the same index are assumed to be correlated, while being independent from rows with different indices. To formalize this, IVA models the columns of the $\bX^{[k]}$ (resp. $\bS^{[k]}$) as independent samples from $\R^N$-valued random vectors $\x^{[k]} = [\sdm{x}_1^{[k]}, \ldots,\sdm{x}_N^{[k]}]^\top$ (resp. $\s^{[k]} = [\sdm{s}_1^{[k]}, \ldots,\sdm{s}_N^{[k]}]^\top$). We can thus rewrite the model in \eqref{matricial mixing equations} using random vector notations: $(\forall k \in \kfirst), \x^{[k]} = \bA^{[k]} \s^{[k]}$. Regrouping the components of the $\s^{[k]}$ with corresponding indices, we obtain $N$ all $\R^K$-valued random vectors, $\s_n = [\sdm{s}_n^{[1]}, \ldots,\sdm{s}_n^{[K]}]^\top$ for $n \in \{1,\ldots,N\}$ called source component vectors, where each entry of a SCV accounts for the corresponding dataset. The goal of IVA is to make the SCVs as independent as possible.

 IVA aims at recovering the sources by building demixing matrices $(\bW^{[k]})_{1 \leq k \leq K}$. At the same times, it builds estimated SCVs $(\y_n)_{1 \leq n \leq N}$ whose distributions have probability density functions respectively denoted $(p_{\y_n})_{1 \leq n \leq N}$ and that we suppose mutually independent. Similarly as for SCVs, we denote $(\forall n \in \nfirst), \y_n = [\sdm{y}_n^{[1]}, \ldots,\sdm{y}_n^{[K]}]^\top$ the estimated SCVs, and we reorganize the components to define $(\forall k \in \kfirst), \y^{[k]} = [\sdm{y}_1^{[k]}, \ldots,\sdm{y}_N^{[k]}]^\top$. Then, we see with \eqref{build source estimates} that for all $k \in \kfirst$, $\x^{[k]}$ is estimated by $(\bW^{[k]})^{-1} \y^{[k]}$, whose probability distribution is entirely determined by the demixing matrices and the $(p_{\y_n})_{1 \leq n \leq N}$. Said otherwise, $\W$ and $(p_{\y_n})_{1 \leq n \leq N}$ yield an estimated generative model for our observed datasets.

 \subsection{IVA-G cost function}




In the following, for every $n \in \{1,\ldots,N\}$, we denote by $\bY_n \in \R^{K \times V}$ the matrix obtained by stacking vertically the $n$-th rows of $\bY^{[k]}$ for $k \in \kfirst$. Therefore, using \eqref{build source estimates}, we have 
\begin{equation}
\label{link between x and y_n}
    (\forall n \in \nfirst) \quad \bY_n = \bW_n \bX,
\end{equation}
where
\begin{equation}
\bW_n =
\begin{pmatrix}
 \bw_n^{[1]\top} & 0 & \ldots & 0\\
 0 & \bw_n^{[2]\top} & \ldots & 0\\
 \vdots & \vdots & \ddots & \vdots\\
 0 & 0 & \ldots & \bw_n^{[K]\top}
\end{pmatrix} \in \R^{K \times KN}
\end{equation}
and 
\begin{equation}
\bX = 
\begin{pmatrix}
    \bX^{[1]}\\
    \vdots\\
    \bX^{[K]}
\end{pmatrix}
\in \R^{KN \times V}.
\end{equation}

The objective of IVA is to determine $\W$ and $(p_{\y_n})_{1 \leq n \leq N}$ that maximize the log-likelihood of $\bX$ in our estimated generative model, given by \cite{Anderson_IVA_IC_PB}
\begin{equation*}
\label{likelihood in general}
    \sum_{n=1}^N \log p_{\y_n}(\bY_n) + V \sum_{k=1}^K \log |\det \bW^{[k]}|.
\end{equation*}

In the Gaussian case, considered here, we model $(\y_n)_{1 \leq n \leq N}$ as centered non-degenerate Gaussian vectors, whose pdf is thus entirely determined by their covariance matrices, or equivalently, their precision matrices that we denote $(\bC_n)_{1\le n \le N}$. As done previously, to simplify our notation, we introduce the tensor $\C = [\bC_1,\ldots,\bC_N] \in (\S_K^{++})^N$ that gathers the estimated precision matrices of all the SCVs. Moreover, we assume sample independence, so the pdf of $\bY_n$ is the product of the pdfs of its columns. Under this Gaussian modeling, we have,
\begin{align}
\label{p_y_n}
&\nonumber (\forall \by \in \R^K)\\
&\log p_{\y_n}(\by) = \frac{1}{2} \log \det \bC_n - \frac{K}{2} \log 2 \pi - \frac{1}{2} \by^\top \bC_n \by.
\end{align}
Hence, the problem becomes equivalent to minimize, with respect to $\W,\C$, the cost function $J_{\rm IVA-G}(\W,\C)$
defined as
\begin{align}
\label{unregularized cost}
&J_{\rm IVA-G}(\W,\C)
+\frac{1}{2} \sum_{n=1}^N \log \det \bC_n+\sum_{k=1}^K \log |\det \bW^{[k]}|
\nonumber\\
&= \frac{1}{2}\sum_{n=1}^N \frac{1}{V} \tr(\bC_n \bY_n \bY_n^\top) 
\nonumber\\
&= \frac{1}{2}\sum_{n=1}^N \tr(\bC_n \bW_n \Rx \bW_n^\top) 
,
\end{align}
 where $\Rx = \frac{1}{V} \bX \bX^\top$ is the empirical covariance matrix of $\hat{\x}$.\\

Note that the domain of function \eqref{unregularized cost} can be extended to $\R^{N \times N \times K} \times \R^{K \times K \times N}$ by setting $J_{\rm IVA-G}(\W,\C) = +\infty$ if $\bW^{[k]}$ is singular for some $k \in \kfirst$ or if $\bC_n$ is not symmetric positive definite for some $n \in \nfirst$. 

\begin{remark}
As we will show in Section \ref{sec:conv}, the non-singularity of $\Rx$ is a sufficient condition to the lower-boundedness of \eqref{unregularized cost}, and as such, to the well-posedness of our optimization problem. In practice, for a large number of samples --- that is $V > KN$ --- drawn from a continuous probability distribution, the empirical covariance matrix $\Rx$ is non-singular almost surely. Hence, we will suppose that this condition holds in the remainder of this article.
\end{remark}

\subsection{IVA-G-V and IVA-G-N approaches}
\label{sec:ivagv}
In \cite{Anderson2012}, the authors proposed two algorithms, called \textbf{IVA-G-V} and \textbf{IVA-G-N}, to solve the IVA-G estimation problem. To do so, they reformulate the problem as the minimization of the following cost function
\cite{Anderson2012}:
\begin{equation}
\label{Anderson cost}
\widetilde{J}_{\rm IVA-G}(\W) = \frac{1}{2} \sum_{n=1}^N \log \det (\bW_n \Rx \bW_n^\top) - \sum_{k=1}^K \log |\det \bW^{[k]}|.
\end{equation}
These approaches implicitly estimate the covariance matrices of the sources by $\bC_n^{-1} = \bW_n \Rx \bW_n^{\top}$, that is, the empirical covariance matrices of the $(\y_n)_{1\leq n \leq N}$. The minimization of \eqref{Anderson cost} is usually performed using either a gradient-based, or a Newton-based solver, leading to \textbf{IVA-G-V} and \textbf{IVA-G-N} schemes, respectively. To the best of our knowledge, these algorithms do not benefit from strong guarantees in terms of convergence of the iterates.

It is easy to show that finding $(\W,\C)$ that minimizes the proposed objective function \eqref{unregularized cost} is actually equivalent to finding $\W$ that minimizes the cost \eqref{Anderson cost} used in state-of-the-art IVA-G approaches, as stated in the following known result. 




\begin{theorem}[Equivalence between $J_{\rm IVA-G}$ and $\widetilde{J}_{\rm IVA-G}$]\
\label{equivalence ML MI}

For a given $\W \in \R^{N \times N \times K}$, if all the $\bW^{[k]}$ are non-singular, then 
\begin{enumerate}
    \item[(i)] $J_{\rm IVA-G}(\W,.)$ is minimized over $\R^{K \times K \times N}$ at
    $\widehat{\C}(\W)  = 
 [\widehat{\bC}_1(\bW_1),\ldots,\widehat{\bC}_N(\bW_N)]$ where
 \begin{equation}
    (\forall n \in \{1,\ldots,N\})\,\, \widehat{\bC}_n(\bW_n) = (\bW_n \Rx \bW_n^{\top})^{-1}.
 \end{equation}
    \item[(ii)] We have, for every $\W \in \R^{N \times N \times K}$:
    \begin{equation}
    \underset{\C \in (\mathcal{S}_K^{++})^N}{\min} J_{\rm IVA-G}(\W,\C) = \widetilde{J}_{\rm IVA-G}(\W) + \frac{KN}{2}.
\end{equation}
\end{enumerate}
 
\end{theorem}
 

Theorem \ref{equivalence ML MI} shows that $\widehat{\W}$ is a minimizer for $\widetilde{J}_{\rm IVA-G}$ if and only if $(\widehat{\W},\widehat{\C}(\widehat{\W}))$ is a minimizer for $J_{\rm IVA-G}$, hence the equivalence. The advantage of relying on the proposed cost function that takes $\C$ as an input block variable is that it offers a more structured form, allowing an efficient use of a block alternating minimization scheme, and benefiting from sounded convergence, as we will show in Section \ref{sec:Palm}.

\subsection{Ambiguities in IVA-G model}
    There are two ambiguities in the IVA-G generative model, that correspond to information on the parameters that cannot be deduced from the observed data. The first one is the \emph{permutation ambiguity}. New mixing matrices can be obtained by renumbering of the SCVs, that is, by defining $\widetilde{\bA}^{[k]} = \bA^{[k]}\sbmm{P}$ with $\sbmm{P}$ a permutation matrix. This change however does not modify the value of the cost function \eqref{unregularized cost}. The second one is the \emph{scaling ambiguity}.    
For every $k \in \kfirst$ and $n \in \nfirst$, and for any $\alpha_n^{[k]} \in \R \setminus \{0\}$, we can replace $\s_n^{[k]}$ with $\widehat{\s}_n^{[k]} = \alpha_n^{[k]} \s_n^{[k]}$ and $a_{m,n}^{[k]}$ with $\hat{a}_{m,n}^{[k]} = (\alpha_n^{[k]})^{-1} a_{m,n}^{[k]}$ for every $m \in \nfirst$. Those transformations let the random vectors $\x_n^{[k]}$ unchanged, and consequently, they do not affect the likelihood expression. Hence, the demixing matrices can only estimate the inverse of the ground truth mixing matrices, up to the permutation and scaling ambiguity. The latter ambiguity moreover raises the problem that a minimizing sequence of $J_{\rm IVA-G}$ is not necessarily bounded. This motivates our proposition for a regularized version for the cost.

\subsection{Proposed regularized cost function}

Let us note, for all $\C \in (\S_K^{++})^N$ and $n \in \nfirst, (c_{n,k,k})_{1 \leq k \leq K} = \diag(\bC_n)$. Due to the scaling ambiguity raised above, for any $(\W,\C)$, it is possible to rescale the coefficients to define $(\W',\C')$ such that $J_{\rm IVA-G}(\W',\C') = J_{\rm IVA-G}(\W,\C)$ and that
\begin{equation}
\label{diagonal constraint}
    (\forall n \in \nfirst)(\forall k \in \kfirst) \quad c'_{n,k,k} = 1.
\end{equation}

To do this, one can set, for each $(k,n)$, $\alpha_n^{[k]} = \frac{1}{\sqrt{c_{n,k,k}}}$. As a consequence, if a minimizer of $J_{\rm IVA-G}$ exists, then there exists at least another minimizer, satisfying $c_{n,k,k} = 1$ for all $(k,n)$. To mitigate this ambiguity, we thus propose to add a quadratic penalty to the cost function to control the distance to $1$ of the diagonal coefficients of the precision matrices.\footnote{In \textbf{IVA-G-V} and \textbf{IVA-G-N}, the scale ambiguity is managed by an ad-hoc renormalizing of the rows of the demixing matrices after each iteration of the minimization solver.} \\

\begin{remark}
Let us notice that a minimizer of $J_{\rm IVA-G}$, such that $c_{n,k,k} = 1$ has no reason to be qualitatively better than any other minimizer of $J_{\rm IVA-G}$. The proposed regularization aims at reducing the number of distinct minima, and at ensuring some mathematical properties of the cost we will leverage  to prove the convergence of our optimization algorithms. It is still possible, once convergence is reached, to rescale the sources a posteriori.
\end{remark}

In addition to the scaling penalty term, we also propose to add an extra term to the cost function, to constrain the singular values of the recovered precision matrices to be positive by a minimum margin. This term aims at avoiding numerical issues that could arise at the boundary of the logarithm determinant definition domain, without damaging the quality of the results. In practice, the constraint is simply imposed, by adding an indicator function $\iota_{[\epsilon,+\infty)^K}$, 
equal to $0$ for non-negative entries, $+\infty$ otherwise.
Our objective is
to impose that the components of the vector
of singular values
$\boldsymbol{\sigma}_{\bC_n}\in \mathbb{R}^K$
of matrix $\bC_n$ are above a certain $\epsilon > 0$, for every $n$. We thus obtain the final form for our proposed (regularized) cost function, denoted by $J_{\rm IVA-G}^{\text{Reg}}$:

\begin{align}
\label{Jiva reg}
&(\forall (\W,\C) \in (\mathcal{GL}_N(\R)^K \times (\S_K^{++})^N))\nonumber\\
    &J_{\rm IVA-G}^{\text{Reg}}(\W,\C) = 
    J_{\rm IVA-G}(\W,\C)
    + \frac{\alpha}{2} 
    \sum_{n=1}^N \|\operatorname{diag}(\bC_n)-\boldsymbol{1}_K\|^2\nonumber\\
    &\qquad\qquad\qquad\;\;+ \sum_{n=1}^N 
 \iota_{[\epsilon,+\infty)^K}
 (\boldsymbol{\sigma}_{\bC_n}),
\end{align}
$\alpha > 0$ is an hyper-parameter that controls the strength of the introduced  regularization.


Our IVA-G method then aims to minimize \eqref{Jiva reg}. This is a challenging non-convex and non-smooth problem. The next section is dedicated to discuss further the properties of \eqref{Jiva reg}, and to design an efficient optimization algorithm to find a critical point of it.


\section{Proposed minimization algorithm}
\label{sec:Palm}

\subsection{Mathematical properties of \texorpdfstring{$J_{\rm IVA-G}^{\text{Reg}}$}{J}}

In order to design an appropriate minimization algorithm for \eqref{Jiva reg}, let us examine the structure and properties of the cost function $J_{\rm IVA-G}^{\text{Reg}}$. First, let us remark that minimizing \eqref{Jiva reg} on $\mathcal{GL}_N(\R)^K \times (\mathcal{S}_K^{++})^N$ is equivalent to 
\begin{equation}
\label{eq:palmpb}
    \operatornamewithlimits{minimize}_{(\W,\C) \in \R^{N \times N \times K} \times \R^{K \times K \times N}} h(\W,\C) + f(\W) + g(\C),
\end{equation}
with
\begin{align}
    2 h(\W,\C) &= \sum_{n=1}^{N}\tr (\bC_n \bW_n \Rx \bW_n^\top)  + \alpha \|\operatorname{diag}(\bC_n)-\boldsymbol{1}_K\|^2 ,\label{eq:hpalm}\\
    f(\W) \ &= 
    \begin{cases}
      & - \sum_{k=1}^{K} \log |\det \bW^{[k]}|\\
      &\quad\text{if $(\forall k \in \kfirst) \, \bW^{[k]} \in \mathcal{GL}_N(\R)$}\\
      & +\infty ~ \text{otherwise}, \label{eq:fPALM}
    \end{cases}\\
    g(\C) &= 
    \begin{cases}
      & - \frac{1}{2} \sum_{n=1}^{N} \log \det \bC_n\\
      &\quad\text{if $(\forall n \in \nfirst) \, \bC_n-\epsilon \bI_K \in  \S^+_K$}\\
      & +\infty ~ \text{otherwise}. \label{eq:gPALM}
    \end{cases}
\end{align}
As we already mentioned, $\epsilon > 0$ serves as obtaining better regularity for the cost function (with closed definition domain). It is typically taken very small (e.g., $\epsilon = 10^{-12}$ in our experiments) to ensure the problem is essentially equivalent to a search for a maximum of the likelihood. A similar strategy was adopted in \cite{chouzenoux_optimal_2019}.
Formulation \eqref{eq:palmpb} has the advantage of isolating a differentiable term $h$ acting on both set of variables $(\W,\C)$, and two non differentiable terms, $f$ and $g$, acting separately on $\W$ or $\C$.

\paragraph{Function $h$}
Let us first study function $h$ acting jointly on $\W$ and $\C$ variables. We state the following lemmas, regarding the expression and smoothness properties, of its partial gradients, with respect to the first or second entry, the other being fixed. 

\begin{lemma}
The partial gradient of $h$ with respect to variable $\W$, evaluated at $(\W,\C) \in \R^{N \times N \times K} \times \R^{K \times K \times N}$, reads:
    \begin{equation}
\label{grad_W J}
\nabla_\W h(\W,\C) = (\nabla_{\bW^{[k]}}h(\W,\C))_{1 \leq k \leq K} \in \R^{N \times N \times K},
\end{equation}
where, for all $k \in \kfirst$,
\begin{align*}
&\nabla_{\bW^{[k]}}h(\W,\C) =
\Big(\frac{\partial h(\W,\C)}{\partial w_{n,m}^{[k]}}\Big)_{1\le n,m\le N}\nonumber\\
= &\begin{pmatrix}
[\bC_1\bW_1\Rx]_{k,(k-1)N+1}& \ldots &
[\bC_1\bW_1\Rx]_{k,(k-1)N+N}\\
\vdots & & \vdots\\
[\bC_N\bW_N\Rx]_{k,(k-1)N+1}& \ldots &
[\bC_N\bW_N\Rx]_{k,(k-1)N+N}
\end{pmatrix}.
\end{align*}
Moreover, for every $\C \in  \R^{K \times K \times N}$,  $\nabla_\W h(\cdot,\C)$ is Lipschitz continuous, with modulus
\begin{equation}
\label{eq:lipW}
L_{\W}(\C) = \rho_{\C} \varrho_{\Rx}, 
\end{equation}
 \label{lem:gradW}
where $\Rx^{[k]}$ is the matrix obtained by extraction of the columns $kN+1$ to $(k+1)N$ of $\Rx$,
\begin{equation}
\label{Rx}
\varrho_{\Rx} = \underset{1 \leq k \leq K}{\max} \|\Rx^{[k]}\|_{\rm S}
\end{equation}
and
\begin{equation}
\label{rho_C}
(\forall \C \in \R^{K \times K \times N}) \quad 
    \rho_{\C} = \underset{1 \leq n \leq N}{\max} \|\bC_n\|_{\rm S}.
\end{equation}
\end{lemma}

\begin{proof}
See Appendix \ref{proof gradW}.
\end{proof}

\begin{lemma}
The partial gradient of $h$ with respect to $\C$, evaluated at $(\W,\C) \in \R^{N \times N \times K} \times \R^{K \times K \times N}$, reads:
    \begin{equation}
\label{grad_C J}
\nabla_\C h(\W,\C) = (\nabla_{\bC_n} h(\W,\C))_{1 \leq n \leq N} \in \R^{K \times K \times N},
\end{equation}
where, for all $n \in \nfirst$,
\begin{align*}
\nabla_{\bC_n} h(\W,\C) = \frac{1}{2} \bW_n \Rx \bW_n^\top + \alpha (\Diag(\bC_n) - \sbmm{I}_K).
\end{align*}
Moreover, for every $\W \in \R^{N \times N \times K}$, $\nabla_{\C}h(\W,\cdot)$ is Lipschitz continuous with constant modulus 
\begin{equation}
L_{\C} = \alpha. \label{eq:lipC}
\end{equation}
 \label{lem:gradC}
\end{lemma}
\begin{proof}
See Appendix \ref{proof gradC}
\end{proof}

According to Lemmas \ref{lem:gradW} and \ref{lem:gradC}, both partial derivatives of $h$ are well defined and continuous with respect to $\W$ and $\C$, which shows in particular that $h$ is a $C^1$ function. 

\paragraph{Functions $f$ and $g$}

Functions $f$ and $g$ are proper (i.e., finite-valued at least at one point), and lower-semicontinuous. 
Furthermore, function $g$ is convex (see, for e.g., \cite[Example 24.66]{Bauschke2017}, for a proof), while function $f$ is not. Both functions $f$ and $g$ are not differentiable but still, it is possible to manipulate them efficiently, for minimization purpose, through their proximity operator \cite{combettes:hal-00643807}.\footnote{See also \url{https://proximity-operator.net/}} The following lemmas provide the expression for these operators, that will then be perused in our algorithm. The proofs for these lemmas mainly rely on the fact that both $f$ and $g$ are so-called spectral functions~\cite{benfenati_proximal_2018}, depending only on the spectral values of their inputs. Note that, despite the non-convexity of $f$, its proximity operator is still uniquely defined, as shown in the proof for Lemma \ref{proxW}.



\begin{lemma}
    \label{proxW}
Let $\W' \in  \R^{N \times N \times K}$, and some $c>0$. We define the proximity operator of $f$ at $\W'$ as 
    \begin{align}
\text{prox}_{c f} (\W' )& = 
\underset{\W \in \R^{N \times N \times K}}{\argmin} \frac{1}{2}\|\W - \W'\|^2 + c f(\W) \nonumber\\
& = (\sbmm{U}_{\bW^{'[k]}}\Diag(\boldsymbol{\sigma}_{\bW^{[k]}})\sbmm{V}_{\bW^{'[k]}})_{1 \leq k \leq K}, \label{prox cf formula}
\end{align}
where $(\forall k \in \{1,\ldots,K\})$
\resizebox{0.28\textwidth}{!}{$\bW^{'[k]} = \sbmm{U}_{\bW^{'[k]}}\Diag(\boldsymbol{\sigma}_{\bW^{'[k]}})\sbmm{V}_{\bW^{'[k]}}$} is the singular value decomposition of $\bW^{'[k]}$, and
\begin{equation}
    \boldsymbol{\sigma}_{\bW^{[k]}} = \frac{\boldsymbol{\sigma}_{\bW^{'[k]}} + \sqrt{\boldsymbol{\sigma}_{\bW^{'[k]}}^2 + 4c}}{2}.
\end{equation}
\end{lemma}

\begin{proof}
    See Appendix \ref{proof proxW}.
\end{proof}

\begin{lemma}
\label{proxC}
    Let $\C' \in \R^{K \times K \times N}$, and some $c>0$. The proximity operator of $g$ at $\C'$ is given by
    \begin{align}
\text{prox}_{c g} (\C' )& =
\underset{\W \in \R^{K \times K \times N}}{\argmin} \frac{1}{2}\|\C - \C' \|^2 + c g(\C) \nonumber \\
& = (\sbmm{U}_{\bC'_n}\Diag(\boldsymbol{\sigma}_{\bC_n})\sbmm{V}_{\bC'_n})_{1 \leq n \leq N}
\label{prox cg formula}
\end{align}

where $(\forall n \in \nfirst), \bC'_n =  \sbmm{U}_{\bC'_n}\Diag(\boldsymbol{\sigma}_{\bC'_n})\sbmm{V}_{\bC'_n}$ is the singular value decomposition of $\bC'_n$, and 
\begin{equation}
    \boldsymbol{\sigma}_{\bC_n} = \max\Bigg(\epsilon, \frac{\boldsymbol{\sigma_{\bC'_n}} + \sqrt{(\boldsymbol{\sigma}_{\bC'_n})^2+2c}}{2}\Bigg).
    \end{equation}

\end{lemma}

\begin{proof}
    See Appendix \ref{proof proxC}.
\end{proof}

\subsection{Proposed \textbf{PALM-IVA-G} algorithm}

As shown in the previous subsection, the minimization of \eqref{Jiva reg} amounts to solve Problem \eqref{eq:palmpb}, that has a structure particularly well suited to block alternating minimization. More precisely, we have shown that the cost function includes (partially) Lipschitz differentiable terms acting on both $(\W,\C)$ variables (Lemmas \ref{lem:gradW} and \ref{lem:gradC}), as well as two terms acting separately on these variables. Despite being non differentiable, these terms are proper, lower-semicontinuous, and with a tractable proximity operator (Lemmas \ref{proxW} and \ref{proxC}). 
These results pave the way for applying a block alternating proximal gradient algorithm, as studied for instance in \cite{chouzenoux2016block,hien_inertial_nodate,chouzenoux2024variational}. Here, we opted for $\textbf{PALM}$ introduced in \cite{bolte2014proximal}, because of its powerful convergence results. We adapted here $\textbf{PALM}$ mechanism to the minimization of the cost \eqref{Jiva reg} and thus designed \textbf{PALM-IVA-G} presented in Alg.~\ref{PALM-IVA-G algorithm}. In $\textbf{PALM}$ initial study, the convergence was shown in the case of sequential block updates. Here, we instead opted for a more versatile update scheme that follows the so-called \emph{essentially cyclic rule} \cite{chouzenoux2016block}, allowing each block to be updated more than once, per main iteration. This assumption hence gives more flexibility to our algorithm, and it is straightforward to adapt the proof given in \cite{bolte2014proximal} to this case, using a similar technique to \cite{chouzenoux2016block}.

\begin{algorithm}[ht!]
\caption{PALM-IVA-G}
\label{PALM-IVA-G algorithm}
\begin{algorithmic}[1]
\raggedright
\Require Empirical covariance $\Rx$, initial tensors $(\W^{(0)}, \C^{(0)}) \in \mathcal{GL}_N(\R)^K \times (\epsilon \bI_K + \mathcal{S}_K^+)^N$,  penalty weight $\alpha>0$, stepsizes $\gamma_\C \in (0,2),\gamma_\W \in (0,1)$, maximal inner/outer loops $\overline{n}_\W \in \N \setminus \{0\}, \overline{n}_\C \in \N \setminus \{0\},\overline{N} \in \N \setminus \{0\}$, precision $\delta>0$.
\Ensure $(\W_{\text{out}},\C_{\text{out}})$
\LComment{Initialization}
\State $\W^{(0,0)} \gets \W^{(0)}$
\State $\C^{(0,0)} \hspace{0.15cm} \gets \C^{(0)}$
\State $c_{\C} \hspace{0.35cm} \gets \frac{\gamma_\C}{\alpha}$ 
\State $i \hspace{0.55cm} \gets 0$
\State $\theta_{\rm ext}^{(0)} \hspace{0.1cm} \gets +\infty$
\LComment{Start Main Loop}
\While{$\theta_{\rm ext}^{(i)} > \delta \textbf{ or } i < \overline{N}$}
\LComment{Update $\W$}
\State $c^{(i)}_{\W} \hspace{0.20cm} \gets \frac{\gamma_\W}{L_{\W}(\C^{(i)})}$ using \eqref{eq:lipW}
\State $j \hspace{0.50cm} \gets 0$
\State $\theta_{\rm int}^{(0)} \hspace{0.10cm} \gets +\infty$
\While{$\theta_{\rm int}^{(j)} > \delta \textbf{ or } j < \overline{n}_\W$}
\State \fontsize{8pt}{8pt} $\W^{(i,j+1)} \gets \prox_{c_\W^{(i)}f} (\W^{(i,j)} - c_\W^{(i)} \nabla_\W h(\W^{(i,j)},\C^{(i)}))$  \State \fontsize{8pt}{8pt} \qquad using \eqref{grad_W J} and \eqref{prox cf formula}
\State \fontsize{8pt}{8pt} $\theta_{\rm int}^{(j)}  \hspace{0.25cm}  \gets \theta_{\W}(\W^{(i,j+1)},\W^{(i,j)})$
\State \fontsize{8pt}{8pt} $j \hspace{0.60cm}  \gets j + 1$
\EndWhile
\State $\W^{(i+1,0)} \gets \W^{(i,j)}$
\State $\W^{(i+1)}  \hspace{0.22cm}  \gets \W^{(i+1,0)}$
\LComment{Update $\C$}
\State $j \hspace{0.50cm} \gets 0$
\State $\theta_{\rm int}^{(0)} \hspace{0.10cm} \gets +\infty$
\While{$\theta_{\rm int}^{(j)} > \delta \textbf{ or } j < \overline{n}_\C$}
\State \fontsize{8pt}{8pt} $\C^{(i,j+1)} \hspace{0.15cm} \gets \prox_{c_\C g} (\C^{(i,j)} - c_\C \nabla_\C h(\C^{(i,j)},\W^{(i+1)}))$
\State \fontsize{8pt}{8pt} $\qquad$ using \eqref{grad_C J} and \eqref{prox cg formula}
\State \fontsize{8pt}{8pt} $\theta_{\rm int}^{(j)}  \hspace{0.25cm}  \gets \theta_{\C}(\C^{(i,j+1)},\C^{(i,j)})$
\State \fontsize{8pt}{8pt} $j \hspace{0.6cm} \gets j+1$
\EndWhile
\State $\C^{(i+1,0)} \gets \C^{(i,j)}$
\State $\C^{(i+1)}  \hspace{0.22cm}  \gets \C^{(i+1,0)}$
\LComment{Evaluate Stopping Criteria}
\State $\theta_\W^{(i+1)} \hspace{0.13cm} \gets \theta_{\W}(\W^{(i+1)},\W^{(i)})$  using \eqref{eq:crit_thetaW}
\State $\theta_\C^{(i+1)} \hspace{0.13cm} \gets \theta_{\C}(\C^{(i+1)},\C^{(i)})$ using \eqref{eq:crit_thetaC}
\State $\theta_{\rm ext}^{(i+1)} \hspace{0.13cm} \gets \max(\theta_{\W}^{(i+1)},\theta_{\C}^{(i+1)})$
\State $i \hspace{0.88cm} \gets i+1$
\EndWhile
\State $(\W_{\text{out}},\C_{\text{out}}) \gets (\W^{(i)},\C^{(i)})$\\
\Return $(\W_{\text{out}},\C_{\text{out}})$
\end{algorithmic}
\end{algorithm}

At each step $i \in \mathbb{N}$ of \textbf{PALM-IVA-G} main loop, we update $\W$ (resp. $\C$) a number $n_\W(i) \leq \overline{n}_\W$ (resp. $n_\C(i) \leq \overline{n}_\C$) of times, with $\overline{n}_\W$ and $\overline{n}_\C$ positive integers. The updates include gradient, and proximal steps, as follows. First, gradient steps on $h$ with respect to the active block, $\W$ or $\C$, with positive stepsizes $c_{\W}^{(i)}$ or $c_{\C}$, respectively, are conducted. Then, proximal steps on $f$ or $g$, are run, using the same stepsizes. Inner and outer loops are controlled by a maximum number of iterations, and furthermore include early stopping tests, based on comparing the following quantities to the (small) precision parameter $\delta>0$: 
\begin{multline}
    \label{eq:crit_thetaW}
    (\forall (\W,\W') \in (\R^{N \times N \times K})^2 ) \\ 
    \theta_{\W}(\W,\W') = 
    \underset{\underset{1 \leq k \leq K}{1 \leq n \leq N}}{\max} \frac{||\bw_n^{'[k]} - \bw_n^{[k]}||^2}{2N},
\end{multline}
\begin{multline}
    \label{eq:crit_thetaC}
    (\forall (\C,\C') \in (\R^{K \times K \times N})^2) \\ 
    \theta_{\C}(\C,\C') = 
    \underset{\underset{1 \leq k \leq K}{1 \leq n \leq N}}{\max} \frac{||\bc'_{n,k} - \bc_{n,k}||^2}{2K},
\end{multline}

where, for $n\in \{1,\ldots,N\}$ and $k\in \{1,\ldots,K\}$, the notation $\bc_{n,k}$ refers to the $k$-th row of the matrix $\bC_n$.

Let us now move to Section \ref{sec:conv}, with the aim to establish the convergence of \textbf{PALM-IVA-G} iterates $(W^{(i)},\C^{(i)})_{i \in \mathbb{N}}$. Practical settings for \textbf{PALM-IVA-G} hyper-parameters will be discussed in Section~\ref{sec:expe}.

\section{Convergence result}
\label{sec:conv}

Let us state our convergence theorem for the proposed \textbf{PALM-IVA-G} algorithm.

\begin{theorem}[Convergence of \textbf{PALM-IVA-G}]
\label{thm:conv}
Assuming the setting of Algorithm \ref{PALM-IVA-G algorithm}, the infinite sequence $(\W^{(i)},\C^{(i)})_{i \in \mathbb{N}}$ converges to a critical point $(\W^*,\C^*)$ of function $J_{\rm IVA-G}^{\text{Reg}}$ given in \eqref{Jiva reg}.
    
\end{theorem}

Here, a critical point is defined as in \cite[Rem.1 (iv)]{bolte2014proximal}, i.e., $0 \in \partial J_{\rm IVA-G}^{\text{Reg}}(\W^*,\C^*)$, where $\partial$ denotes the limiting subdifferential operator. The proof of the above result relies on \cite[Theorem 1]{bolte2014proximal} in the case of a cyclic update of the blocks. The latter states the convergence of a generic form of $\textbf{PALM}$ method under several assumptions regarding the properties of the cost function, and provided the sequence is bounded. It is hence sufficient to show that such conditions hold in our case, to prove Theorem \ref{thm:conv}. In the previous section, we have already seen that $f$ and $g$ are proper and lower semi-continuous functions such that the proximal operators $\prox_{cf}$ and $\prox_{cg}$ are defined for all $c>0$, at any point $\W \in \mathcal{GL}_N(\R)^K$ and $\C \in (\mathcal{S}_K^{++})^N$,  respectively. We also outlined that $h$ is a $C^1$ function and for every given $\W' \in \R^{N \times N \times K}$ (resp. $\C' \in \R^{K \times K \times N}$), the function $\C \mapsto h(\W',\C)$ (resp. $\W \mapsto h(\W,\C')$) is $C^{1,1}_{L_{\C}(\W')}$, i.e. its partial gradient $\C \mapsto \nabla_{\C} h(\W',\C)$ is globally Lipschitz with modulus $L_{\C}(\W')$ (resp. $C^{1,1}_{L_{\W}(\C')}$). We note that those first properties were necessary to well-defining the algorithm itself.

Using our notation, to complete the proof, we need to demonstrate each item of the following Assumption \ref{as:Palm} is satisfied.

\begin{assumption}\ 
\label{as:Palm}
\begin{enumerate}
    \item $\inf_{\R^{N \times N \times K} \times \R^{K \times K \times N}} J_{\rm IVA-G}^{\text{Reg}} > - \infty$.
    \item $\nabla h$ is Lipschitz continuous on bounded subsets of $\R^{N \times N \times K} \times \R^{K \times K \times N}$.
    \item $(\W^{(i)},\C^{(i)})_{i \in \N}$ is bounded.
    \item There exists $(\lambda_{\W}^+, \lambda_{\W}^-, \lambda_{\C}^+, \lambda_{\C}^-) > 0$ such that: $$(\forall i \in \N) \, \lambda_{\W}^- \leq L_{\W}(\C^{(i)}) \leq \lambda_{\W}^+,$$ and $$(\forall i \in \N)\, \lambda_{\C}^- \leq L_{\C}(\W^{(i)}) \leq \lambda_{\C}^+.$$
    \item $J_{\rm IVA-G}^{\text{Reg}}$ is a KL function.
\end{enumerate}
\end{assumption}

Indeed, following the steps of \cite{bolte2014proximal}, the update scheme of the algorithms makes the sequence of the costs decreasing, and using the first item of the above assumption, this sequence has a finite limit. Then, items 2) to 4) enable to prove that the limit set of $(\W^{(i)},\C^{(i)})_{i \in \N}$ is nonempty, compact, and contains only critical points of $\Jivareg$. Finally, the last item is used to prove that $(\W^{(i)},\C^{(i)})_{i \in \N}$ is a Cauchy sequence, hence convergent.
The verification of Assumption~\ref{as:Palm} is technically involved and is therefore provided in Appendix \ref{covergence proof}.

\section{Experimental results}
\label{sec:expe}
We now present a set of experiments, to assess the quantitative, qualitative, and computational performance of $\textbf{PALM-IVA-G}$ on tasks of independent Gaussian sources separation.

\subsection{Experimental protocol}

\subsubsection{Qualitative evaluation}

In all the experiments, the proposed method $\textbf{PALM-IVA-G}$, as well as the benchmarks, are evaluated by means of the so-called 
jISI (joint Intersymbol Interference) score, also used in~\cite{Anderson2012}. This score is an extension of the ISI score that was introduced in \cite{amari_new_1995} to assess the performance of ICA methods. jISI score measures the correspondance between demixing matrices $\W = [\bW^{[1]},\ldots,\bW^{[K]}]$ and the ground truth mixing matrices $\mathcal{A} = [\bA^{[1]},\ldots,\bA^{[K]}]$, up to a common permutation and a rescaling of their rows. Let us note
\begin{equation}
    (\forall k \in \kfirst) \quad \mathbf{G}^{[k]} = \bW^{[k]}\bA^{[k]}
\end{equation}
and, for every $(m,n) \in \nfirst^2$, the mean $\overline{g}_{n,m} = \sum_{k=1}^K |g_{n,m}^{[k]}|$ of the $(n,m)$-th entry of tensor $\mathcal{G} = [\mathbf{G}^{[1]},\ldots,\mathbf{G}^{[K]}]$. The jISI score of the pair $(\W,\mathcal{A})$ is
\begin{align}
&\text{jISI}(\W,\mathcal{A})\nonumber = \frac{1}{2N(N-1)} \Bigg[ \sum_{n=1}^N \Bigg( \sum_{m=1}^N \frac{\overline{g}_{n,m}}{\max_p \overline{g}_{n,p}} - 1 \Bigg) \nonumber\\
&+ \sum_{m=1}^N \Bigg( \sum_{n=1}^N \frac{\overline{g}_{n,m}}{\max_p \overline{g}_{p,m}} - 1 \Bigg) \Bigg]. 
\label{eq:jISI}
\end{align}

As defined in \eqref{eq:jISI}, the jISI score is a real number between $0$ and $1$, with jISI score equals $0$ if and only if $\overline{\mathbf{G}} = (\overline{g}_{n,m})_{1\leq n, m \leq N}$ has exactly one positive coefficient by row and by column, that is when it is a (possibly permuted) diagonal matrix. This happens when the $(\bW^{[k]}\bA^{[k]})_{1 \leq k \leq K}$ are simultaneously permuted diagonal matrices, which is the best situation one can expect from a source separation step. Hence, the smallest the jISI score, the better quality for the source separation.

\subsubsection{Benchmark methods}
Our algorithm $\textbf{PALM-IVA-G}$ is compared against two state of the art algorithms for independant Gaussian source separation, namely \textbf{IVA-G-V} and \textbf{IVA-G-N}, both introduced in \cite{Anderson2012}. Those two algorithms perform the minimization of the cost function $\widetilde{J}_{\rm IVA-G}$ \eqref{Anderson cost} (which, as we remind, only depends on variable $\W$) using, respectively, a gradient descent and a Newton's method. Both implement a backtracking linesearch, and a normalization of the rows of the demixing matrix at each iteration. Hyper-parameters are set as recommended in \cite{Anderson2012}. Let us emphasize that both of these algorithms are empirical, and, to our knowledge, do not benefit from established convergence guarantees, though in practice we did not observe any failure. 

For every experiment, and each algorithm, we obtain an estimate for which we record the jISI score reached at convergence (i.e., when stopping criterion is reached), and the computational time, in seconds, denoted $\texttt{T}$, to reach this convergence point. All algorithms are implemented in Python 3.10.12 and run on a Dell Precision 5820 Workstation {with 11th Gen} Intel(R) Core(TM) i9-10900X at 3.00GHz, equipped with 32Go Ram. 

\subsubsection{Synthetic dataset generation}
For each experimental setup, we define a generative model defined by the ground-truth variables $(\mathcal{A},\mathbf{\Sigma})$, and use this model to generate source data $\S$ of length $V = 10000$, and then mix it into observed data $\X$. The goal is to recover estimates of $([(\bA^{[1]})^{-1},\ldots,(\bA^{[K]})^{-1}],[\mathbf{\Sigma}_1^{-1},\ldots,\mathbf{\Sigma}_N^{-1}])$ up to the scaling and permutation ambiguities. All trials are repeated over $100$ independent runs, and we compute the mean $\mu_{\rm jISI}$ (resp. $\mu_{\texttt{T}}$) and standard-deviation $\sigma_{\rm jISI}$ (resp. $\sigma_{\texttt{T}}$) values for the jISI score (resp. computational time $\texttt{T}$). Most experiments are conducted for various values for the dimensions $(K,N)$, specified in each test case. Depending on the nature of the phenomenon modeled, the covariance and the mixing matrices may have various properties, leading us to define several sets of experiments, detailed hereafter.

We aim at exploring the impact of the overall level of correlation across the datasets (given by the extra-diagonal coefficients of the SCVs covariance matrices), and the variability of those correlations. Our generative model used to simulate SCVs depends on two parameters $\boldsymbol{\rho} = [\rho_1,\ldots,\rho_N]^\top \in [0,1]^N$ and $\lambda \in [0,1]$. Given those, we compute a ground truth tensor $\mathbf{\Sigma} = \left [ \mathbf{\Sigma}_1,\ldots,\mathbf{\Sigma}_N\right ]$ with, for every $n$,
\begin{equation}
    \mathbf{\Sigma}_n = \rho_n \mathbf{1} \mathbf{1}^\top + \frac{\lambda}{R} \bQ \bQ^\top + \eta_n \bI_K
\end{equation}
with
\begin{equation}
    \eta_n = 1 - \rho_n - \lambda \in [0,1],
\end{equation}
matrix $\bQ$ randomly drawn in $\R^{K \times R}$ with elements $q_{i,j} \sim \mathcal{N}(0,1)$ mutually independent, and $R \in \N \setminus \{0\}$ a predefined rank value. 




Then, we have, for every $n$, and every entry $(i,j)$,
\begin{equation}
    \begin{cases}
        \mathbb{E}\{(\mathbf{\Sigma}_n)_{i,j}\} = \rho_n + (\lambda + \eta_n) \delta_{i,j}\\
        \operatorname{Var}\{(\mathbf{\Sigma}_n)_{i,j}\} = \frac{\lambda^2}{R}(1+\delta_{i,j}).
    \end{cases}
\end{equation}
This means that $\boldsymbol{\rho}$ controls the average correlation across the datasets while $\lambda$ controls the variability between the correlations. The third term ensures that the covariance matrices we use are positive definite. In our experiments, we opt for four scenarios, corresponding to low/high variability, and low/high correlation, as defined below:
\begin{itemize}
    \item Case A: low correlation, low variability. $\lambda = 0.04$ and, $\boldsymbol{\rho}$ regularly sampled in $[0.2,0.3]^N$,
    \item Case B: low correlation, high variability. $\lambda = 0.25$ and, $\boldsymbol{\rho}$ regularly sampled in $[0.2,0.3]^N$,
    \item Case C: high correlation, low variability.
    $\lambda = 0.04$ and, $\boldsymbol{\rho}$ regularly sampled in $[0.6,0.7]^N$,
    \item Case D: high correlation, high variability.
    $\lambda = 0.25$ and, $\boldsymbol{\rho}$ regularly sampled in $[0.6,0.7]^N$,
\end{itemize}
The sources are expected to be easier to separate in case D, while case A is more challenging, since increasing the correlation or the variability of the sources decreases the Cramer-Rao Lower-Bound of the ML estimator and thus tends to improve the separation~\cite{Anderson_IVA_IC_PB}. Each of the four cases is investigated, for various sizes $K \in \{5,10,20\}$ and $N \in \{10,20\}$, and we set $R = K+10$. In all experiments, the ground truth tensor $\mathcal{A}$ is simulated by drawing its entries independently from a zero-mean Gaussian distribution with a standard deviation of $1$. 







For each simulated pair $(\mathcal{A},\mathbf{\Sigma})$, we build the corresponding input data $\mathcal{X}$, following the mixing model \eqref{matricial mixing equations}. Note that, as recommended in \cite{Anderson2012}, we systematically performed a whitening of the data, before applying the algorithms. This step amounts to multiplying, for every $k \in \kfirst$, the latent sources by full-rank matrices $(\sbmm{B}^{[k]})_{1 \leq k \leq K}$, and solving the demixing problem on $\bX'^{[k]} = \sbmm{B}^{[k]}\bX^{[k]} = \sbmm{B}^{[k]}\bA^{[k]}\bS^{[k]}$. The whitening matrices are set to decrease the spectral norm of $\Rx$, with the aim to improve the conditioning of the loss function, and hence accelerate the empirical convergence of the methods. 


\subsubsection{Algorithmic settings}
The implementation of the proposed 
\textbf{PALM-IVA-G} algorithm requires the setting of (i) the stepsizes, (ii) the regularization weight, and (iii) the stopping conditions for internal and external loops. 

We set the stepsizes to constant values $(\gamma_\C,\gamma_\W) = (1.99,0.99)$, i.e. as large as possible to meet the convergence theorem assumptions. The penalty parameter $\alpha$ appears to have little influence over the performance of \textbf{PALM-IVA-G}, in terms of jISI metric and computational time, as long as it is chosen in a reasonable range. This can be observed empirically in Fig. \ref{fig:compare alphas} summarizing results for Case A and various values of $\alpha$. Indeed, except for $\alpha = 10$ which seems to yield systematically slower computations, we cannot observe that a value of $\alpha$ gives consistently better results than the others. Similar behaviors were obtained for Cases B to D. For the experiments, we thus retain $\alpha = 1$, as it achieves a good compromise in terms of time complexity.

\begin{figure}[t]
    \centering
    \begin{tabular}{@{}c@{}c@{}c@{}}
\includegraphics[width=4.5cm]{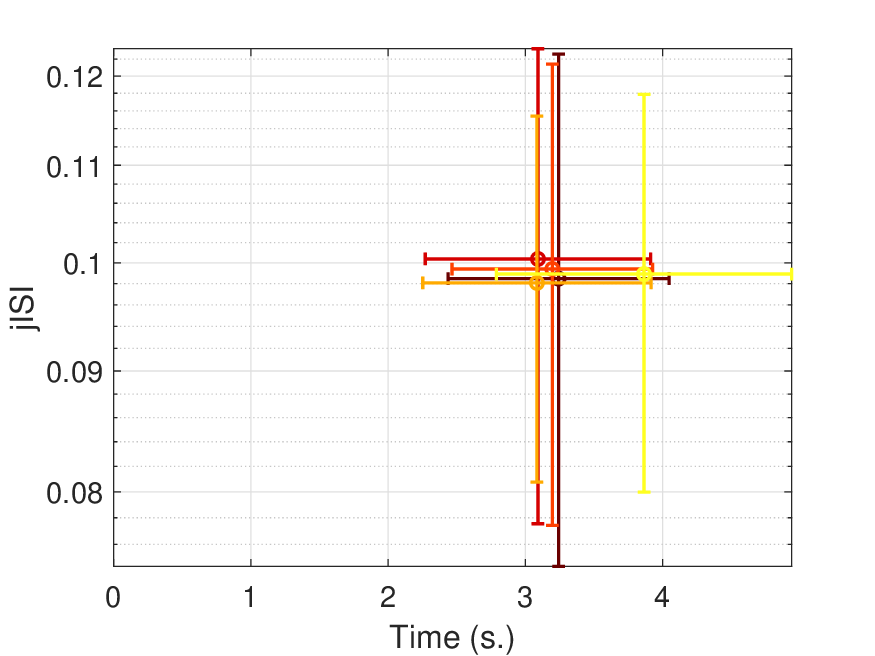}&
\includegraphics[width=4.5cm]{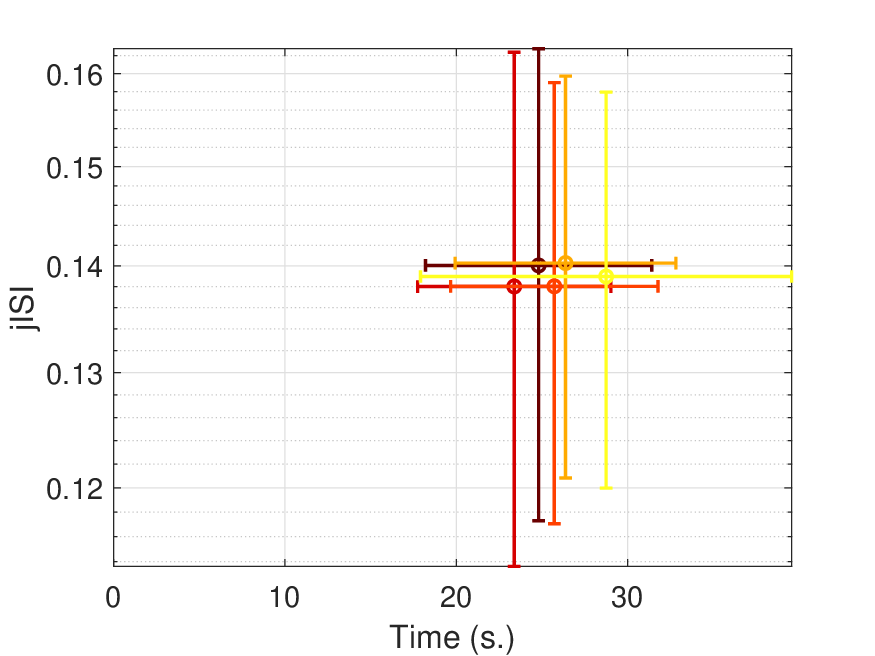}\\
$(K,N) = (5,10)$ & $(K,N) = (5,20)$\\
\includegraphics[width=4.5cm]{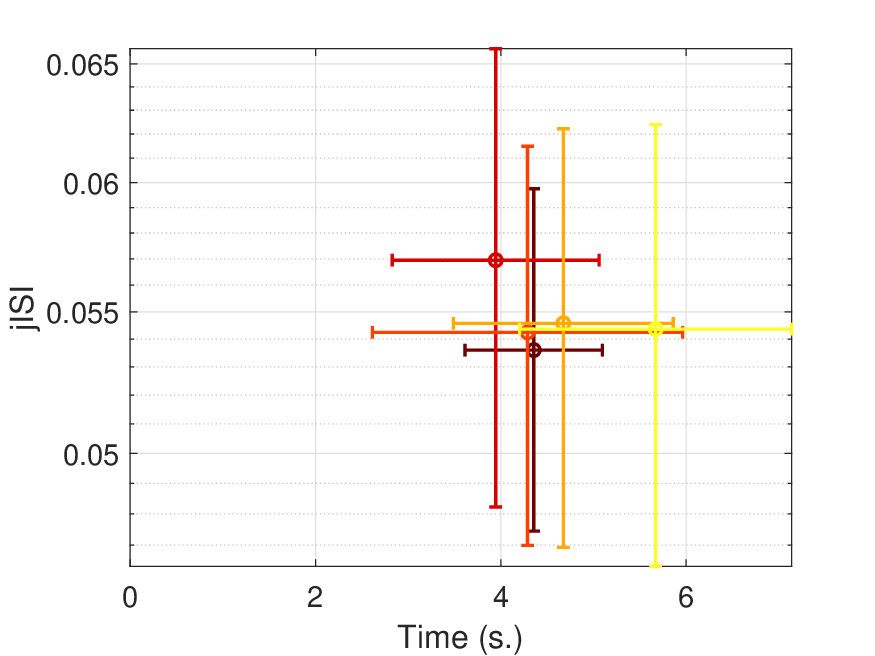}&
\includegraphics[width=4.5cm]{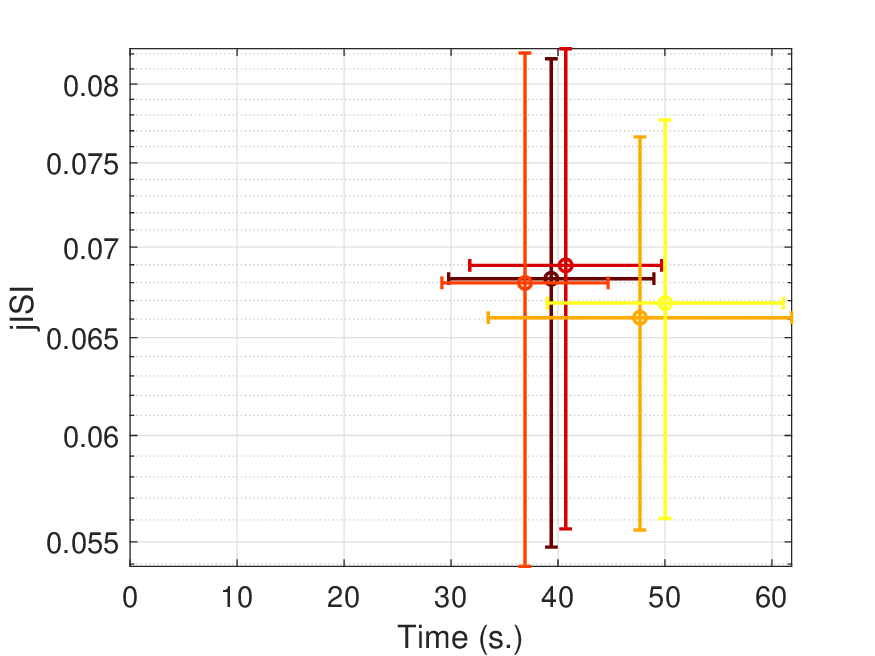}\\
$(K,N) = (10,10)$ & $(K,N) = (10,20)$\\
\includegraphics[width=4.5cm]{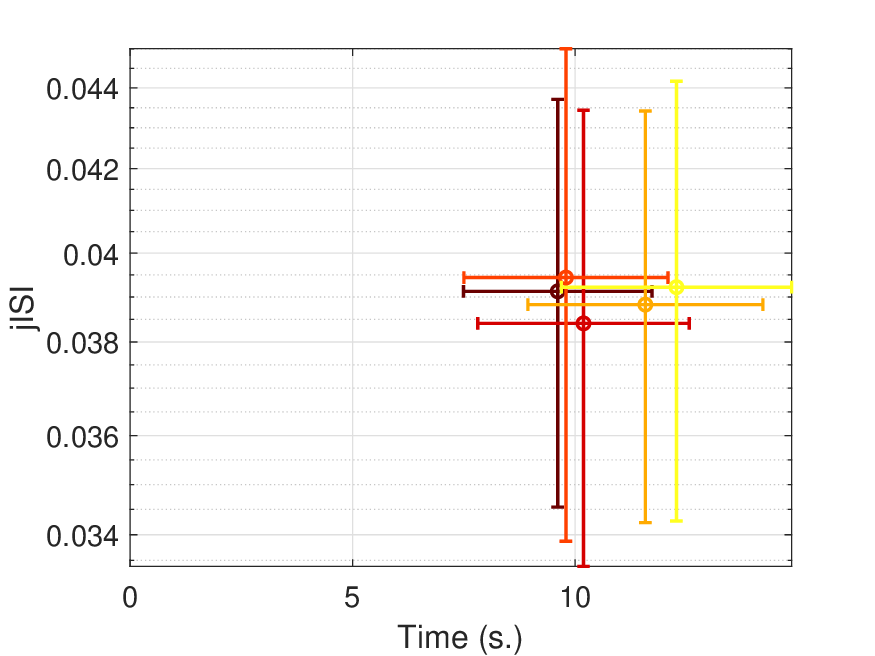}&
\includegraphics[width=4.5cm]{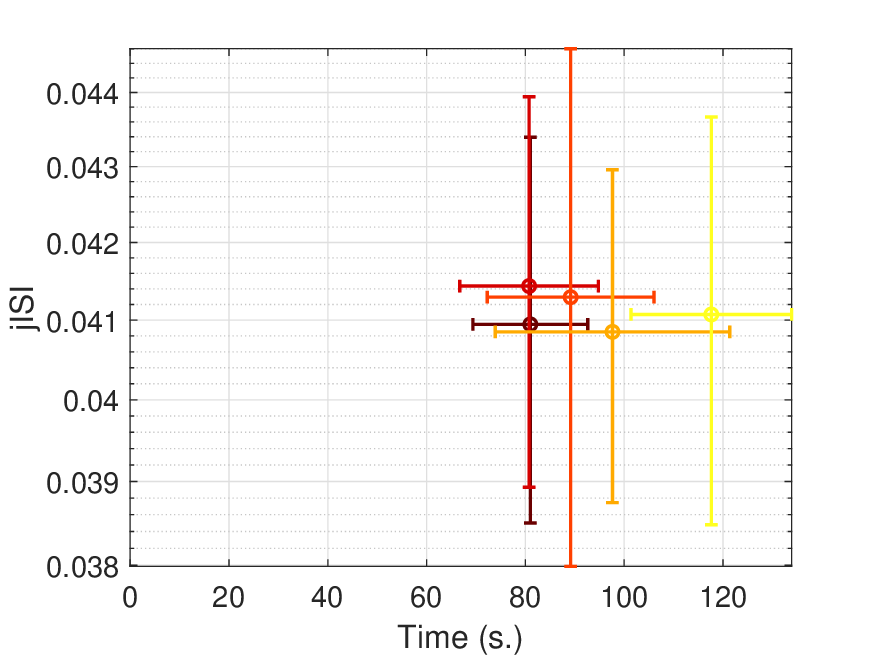}\\
$(K,N) = (20,10)$ & $(K,N) = (20,20)$\\
    \end{tabular}
    \caption{jISI scores vs computation times (mean $\pm$ standard deviation) using \textbf{PALM-IVA-G} for Case $A$ over 20 runs, with $\alpha$ taking values in $\{0.1,0.5,1,5,10\}$ (the brighter color, the higher $\alpha$). The jISI score varies little with $\alpha$, the latter impacting mostly the time, with good compromise at $\alpha = 1$ (dark orange).}
    \label{fig:compare alphas}
\end{figure}

\textbf{PALM-IVA-G} algorithm, as well as its competitors \textbf{IVA-G-V} and \textbf{IVA-G-N} are run until a certain stopping criterion is reached, with a maximum number of $\overline{N} = 20000$ iterations. The precision threshold is set to $\delta = 10^{-10}$, and the maximum number of iterations within the internal loops of \textbf{PALM-IVA-G} are $\bar{n}_\C = 1$ and $\bar{n}_\W = 15$. For \textbf{IVA-G-V} (resp. \textbf{IVA-G-N}), we monitor only the value of \eqref{eq:crit_thetaW} between consecutive iterates, stopping once lower than $\delta = 10^{-6}$ (resp. $\delta = 10^{-7}$). Note that the values of $(\bar{n}_\C,\bar{n}_\W,\delta)$ have been empirically set to reach the best trade-off between jISI score and computational time, ensuring fair comparisons of the methods.

The validity of our settings can be assessed visually on Fig. \ref{fig:empirchart}, showing the cost function evolution using our \textbf{PALM-IVA-G} algorithm in a representative example. 

\begin{figure}[H]
    \centering
    \begin{tabular}{@{}c@{}c@{}}
\includegraphics[width=4.2cm]{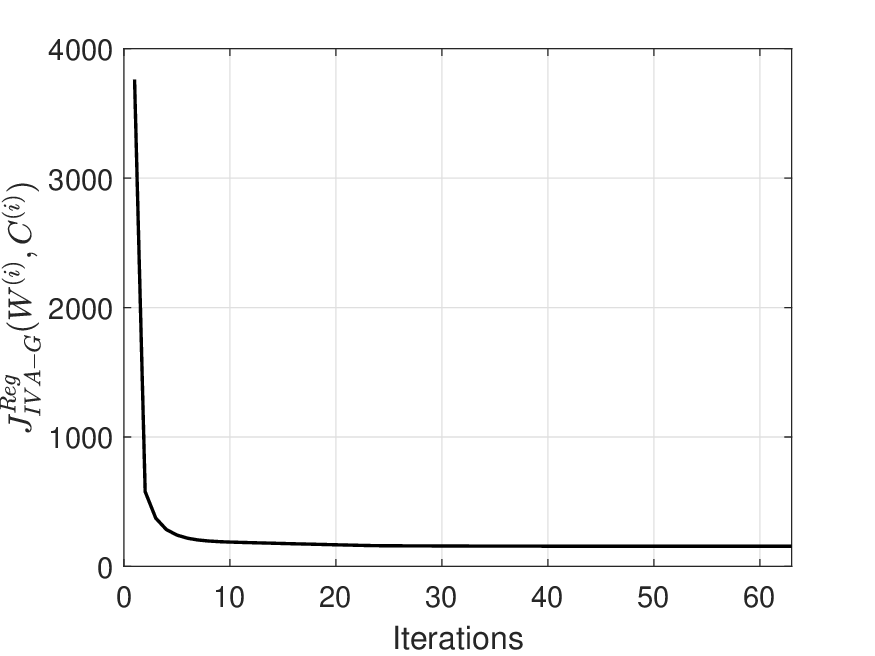}&
\includegraphics[width=4.2cm]{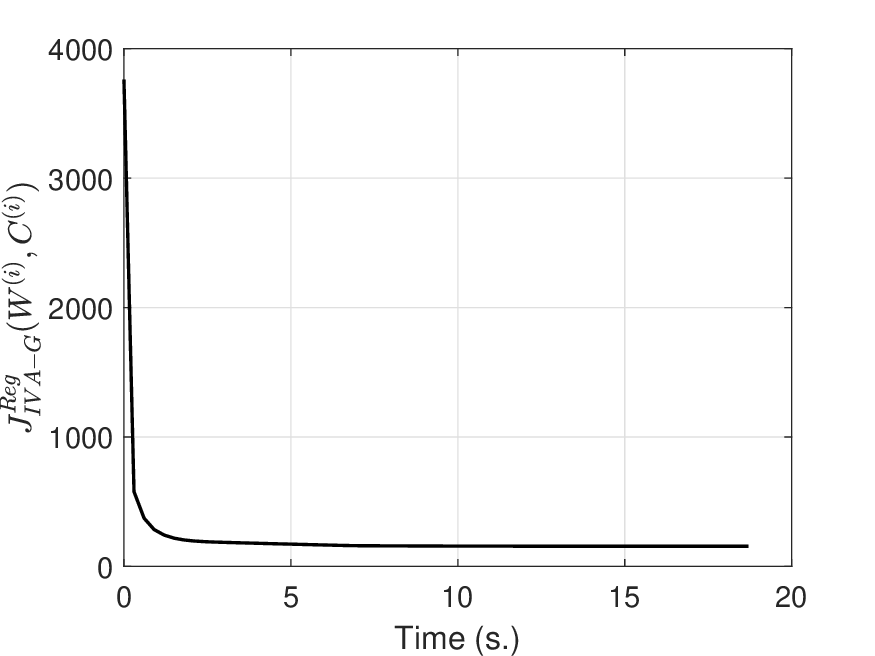} \\
\includegraphics[width=4.2cm]{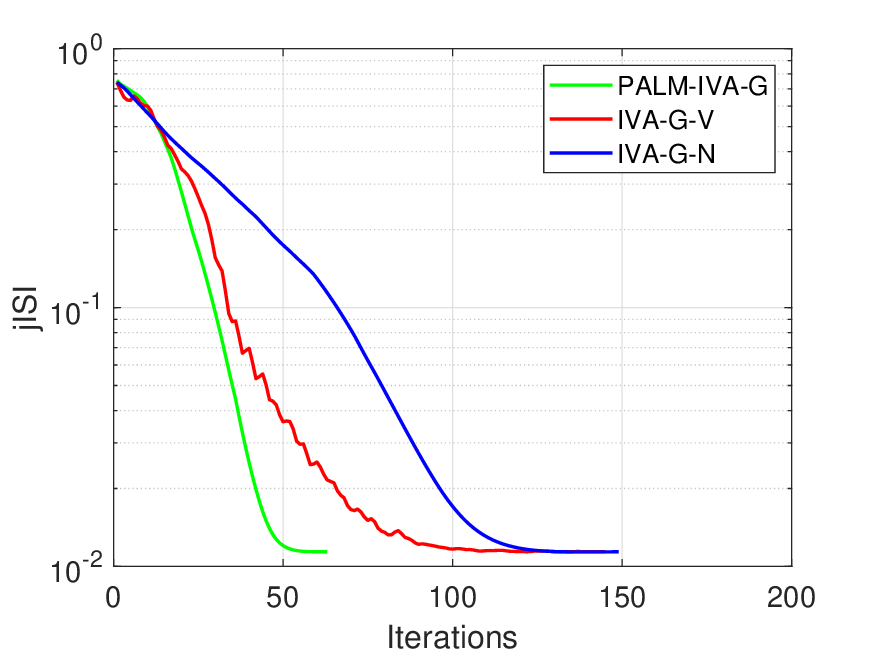}&
\includegraphics[width=4.2cm]{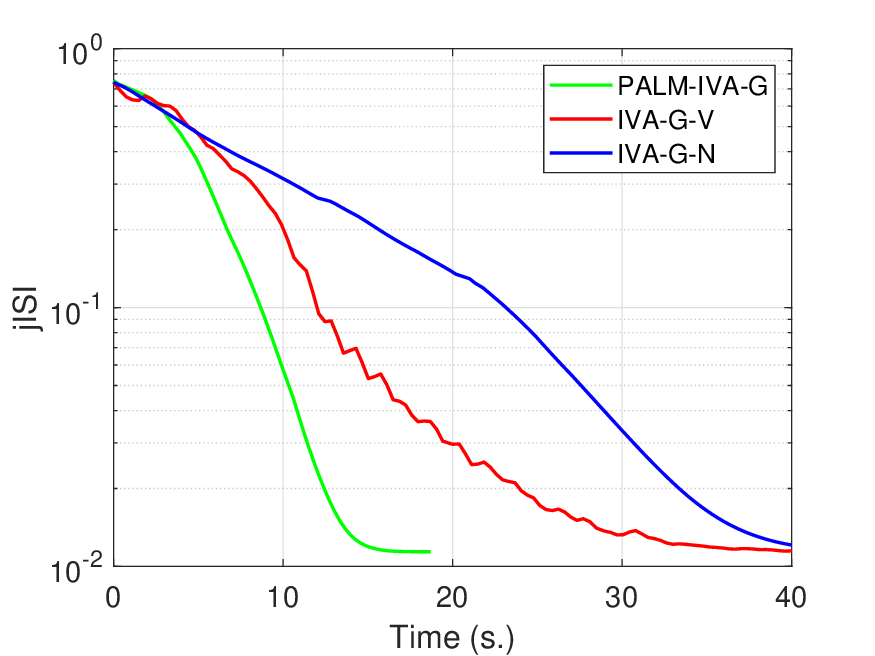}
    \end{tabular}
    \caption{Top: Empirical convergence of \textbf{PALM-IVA-G}, on a synthetic example from Case D, with $(K,N)=(20,20)$, $\lambda = 0.25$ and $\boldsymbol{\rho} \in[0.2,0.3]^\top$. Cost function across iterations (left) and time in second (right). Bottom: On the same example, evolution of jISI score along iterations (left) and time in seconds (right), for the compared methods.}
    \label{fig:empirchart}
\end{figure}

\subsection{Results}

We now present the results of the experiments, and comment on the strengths and weaknesses of our algorithm in comparison with the benchmarks. The results are illustrated in Figs. ~\ref{fig:caseAtoB} and ~\ref{fig:caseCtoD}, and summarized in  Tab. ~\ref{tab:caseAtoD}. On Fig.~\ref{fig:caseAtoB} and ~\ref{fig:caseCtoD}, we display, for each of the three algorithms, a cross centered at $(\mu_\texttt{T},\mu_{\rm jISI})$, and spread out by $\pm \sigma_\texttt{T}$ (horizontal) and $\pm \sigma_{\rm jISI}$ (vertical) axis. The best results thus correspond to a cross located on the bottom left side of the figure (i.e., low jISI score reached in a minimal time). In Tab.~\ref{tab:caseAtoD}, we highlight in bold font (resp. italic bold font) the best results in terms of $\mu_{\rm jISI}$ (resp. $\mu_\texttt{T}$), considering of similar quality jISI scores with less than $10^{-4}$ difference (resp. computation times with less that $10^{-2}$ difference).

All algorithms achieve what appears to be optimal separation in easy cases (B and D). In difficult cases, it is generally either $\textbf{IVA-G-V}$ or $\textbf{PALM-IVA-G}$ that have the best jISI, with an advantage for $\textbf{PALM-IVA-G}$ in small dimensions and an advantage for $\textbf{IVA-G-V}$ in larger dimensions, but in all cases performance remains close. $\textbf{IVA-G-N}$ is generally not as good, but manages to keep its performance close to that of the other algorithms, except for Case A, where it gives jISI values 20 to 50 percent higher than the other algorithms on average.

Computations of $\textbf{PALM-IVA-G}$ are tractable, and stay under two minutes. The running time is sensitive to the number of sources, and seems to grow linearly with the number of datasets within the tested dimensions. For $N = 10$, our algorithm takes less than $15$ seconds to run in average in all cases. Meanwhile, $\textbf{IVA-G-V}$ is less sensitive to $N$ and also manages to separate the sources in two minutes in average at most for the dimensions in the experiment. However, $\textbf{IVA-G-N}$ is the slowest of all three algorithms, it takes several tens of seconds in small dimensions for Cases A and C, and its computational cost becomes prohibitive in larger dimensions, taking an average time of fifteen minutes to run in Case A for $K = 20$ and $N = 20$. Visually, it leads to blue crosses often positioned on the right side of the plots in Figs. ~\ref{fig:caseAtoB} and ~\ref{fig:caseCtoD}.

In contrast to gradient descent, iterations using Newton's method or proximal gradient are more informed and are expected to find the local minimum more efficiently. On the other hand, these methods are more costly, since they involve Hessian inversions for $\textbf{IVA-G-N}$, and SVDs for our $\textbf{PALM-IVA-G}$, and such cost is not necessarily offset by a gain in number of iterations. Besides, the update scheme of $\textbf{PALM-IVA-G}$ implies many more updates of the block $\W$ than the block $\C$, so most of the matrices for which we compute the SVD are of size $N \times N$ rather than $K \times K$, this is why the computation time of $\textbf{PALM-IVA-G}$ increases faster with respect to $N$ than $K$.



\begin{figure}[ht!]
    \centering
    \scalebox{0.95}{
    \begin{tabular}{@{}c@{}c@{}c@{}}
\includegraphics[width=3cm]{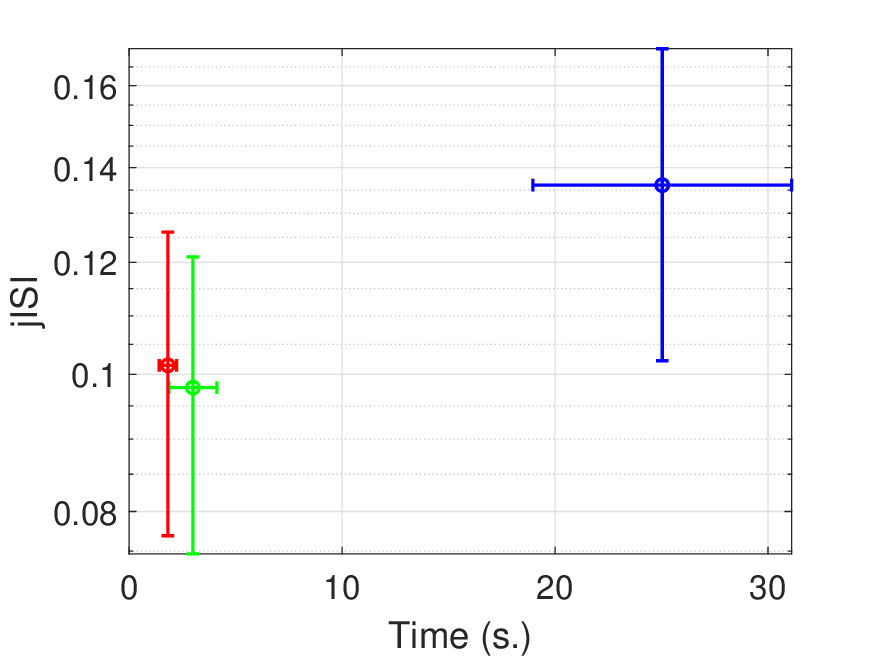}&
\includegraphics[width=3cm]{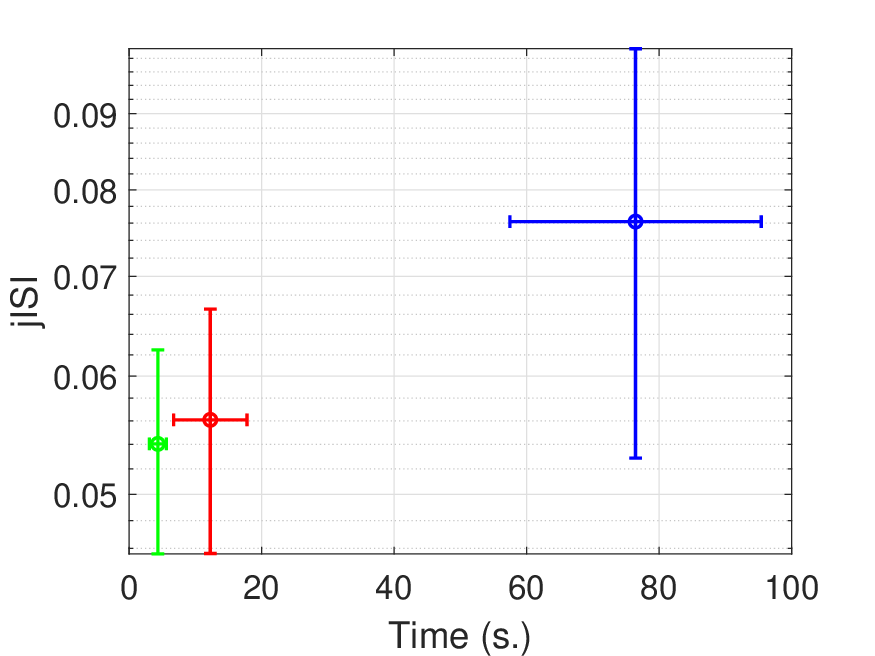}&
\includegraphics[width=3cm]{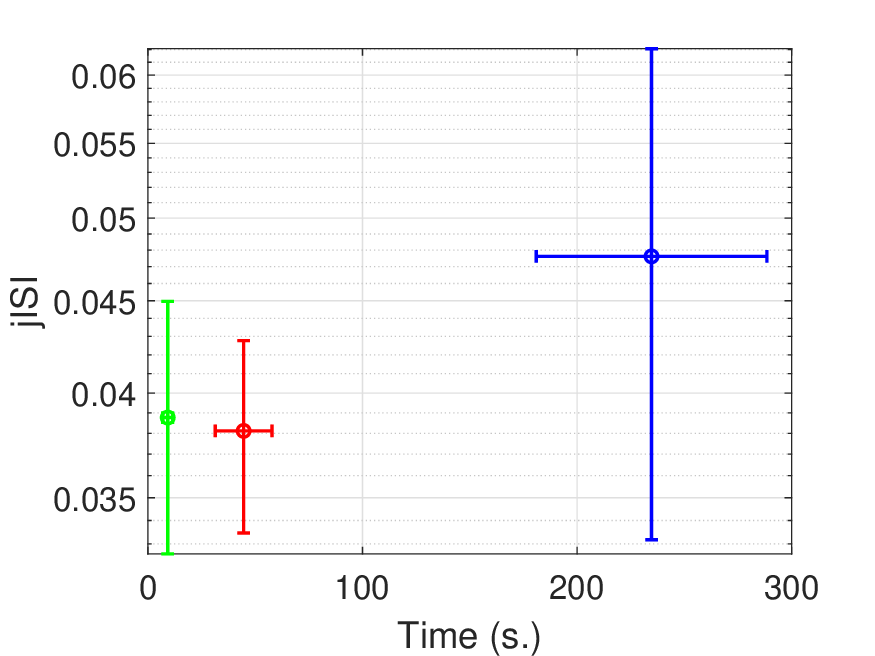}\\
$(K,N) = (5,10)$ & $(K,N) = (10,10)$ & $(K,N) = (20,10)$\\
\includegraphics[width=3cm]{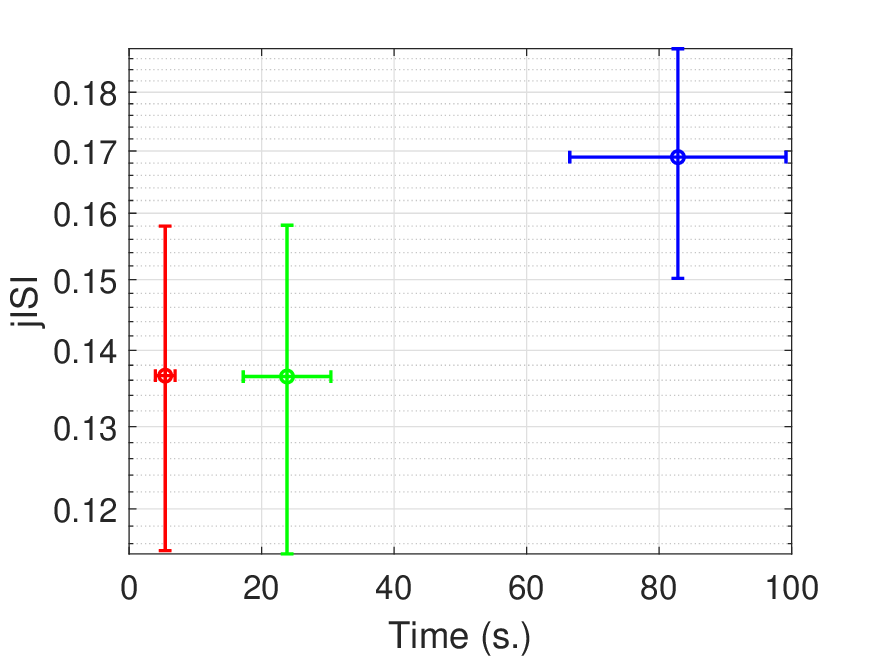}&
\includegraphics[width=3cm]{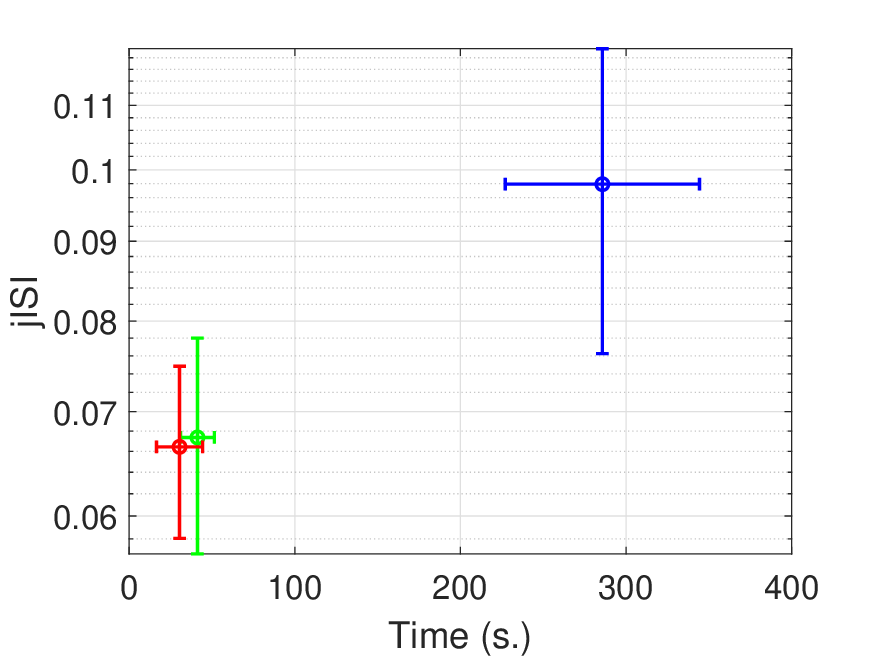}&
\includegraphics[width=3cm]{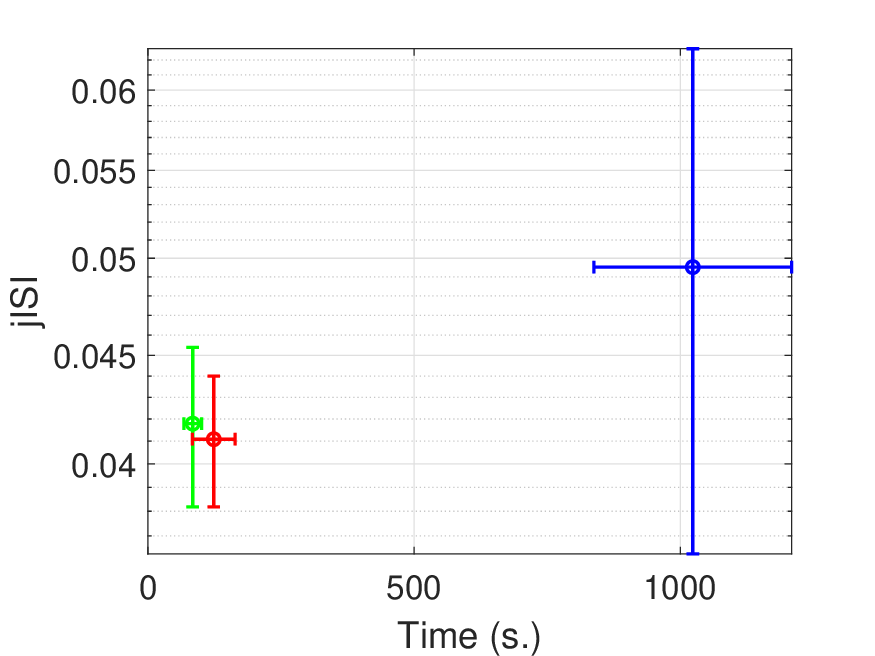}\\
$(K,N) = (5,20)$ & $(K,N) = (10,20)$ & $(K,N) = (20,20)$\\
\multicolumn{3}{c}{\textbf{Case A}:  $\boldsymbol{\rho} \in [0.2,0.3]^N$ and $\lambda = 0.04$} \\
\includegraphics[width=3cm]{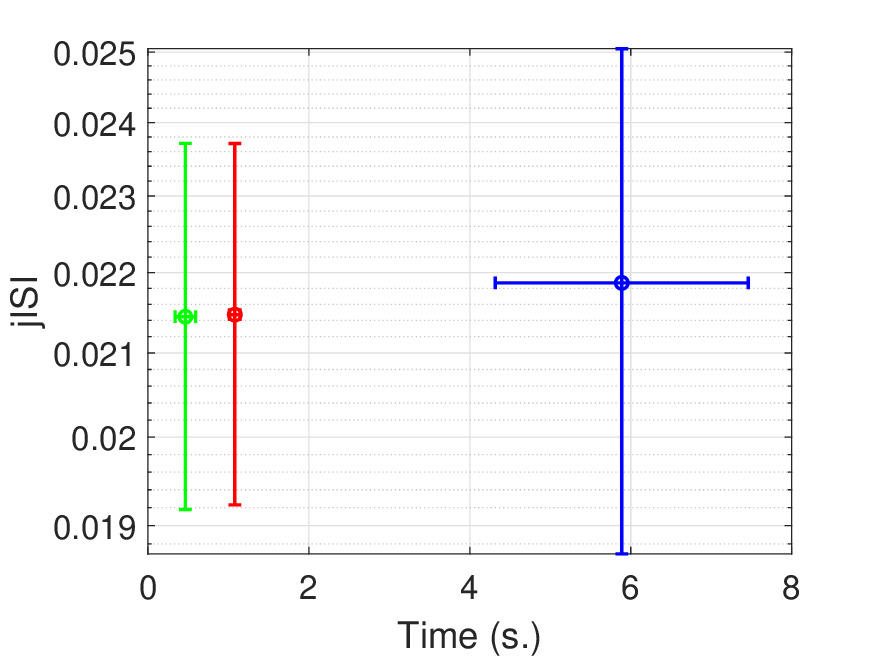}&
\includegraphics[width=3cm]{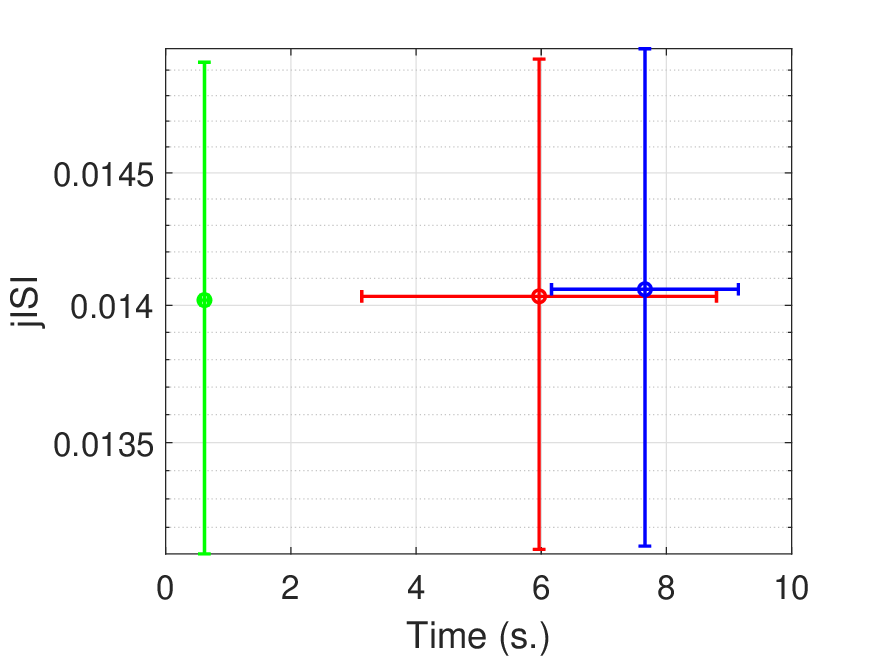}&
\includegraphics[width=3cm]{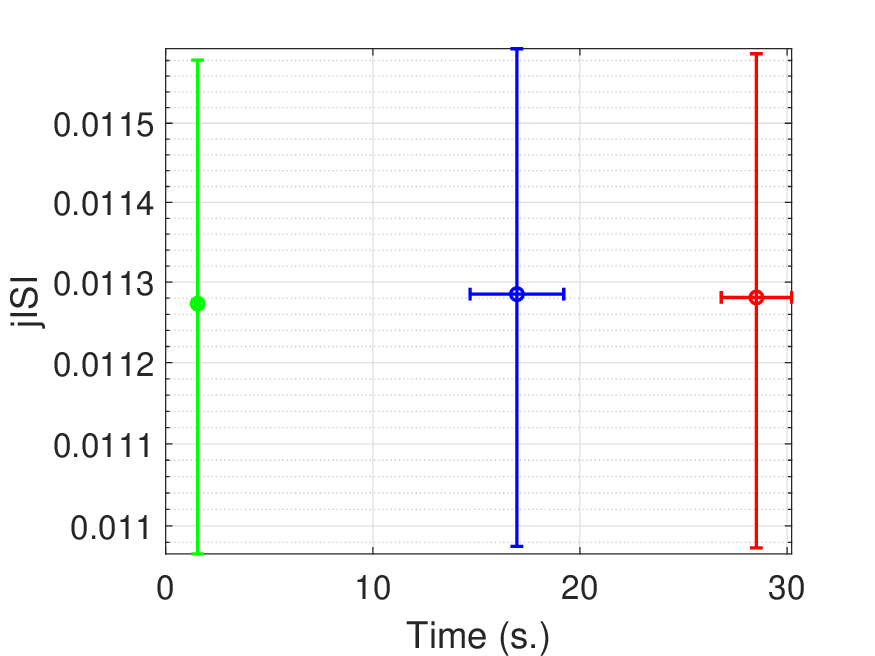}\\
$(K,N) = (5,10)$ & $(K,N) = (10,10)$ & $(K,N) = (20,10)$\\
\includegraphics[width=3cm]{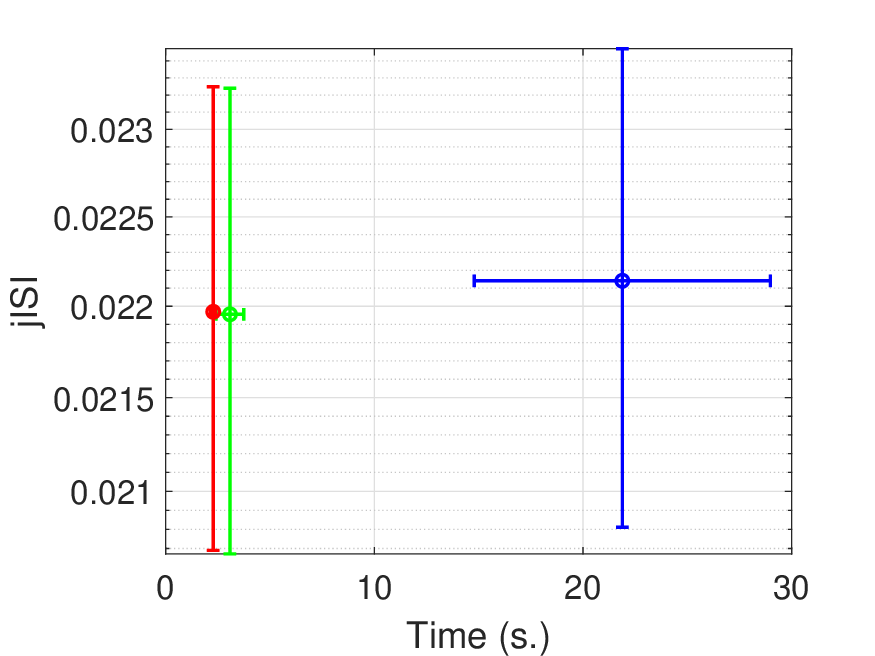}&
\includegraphics[width=3cm]{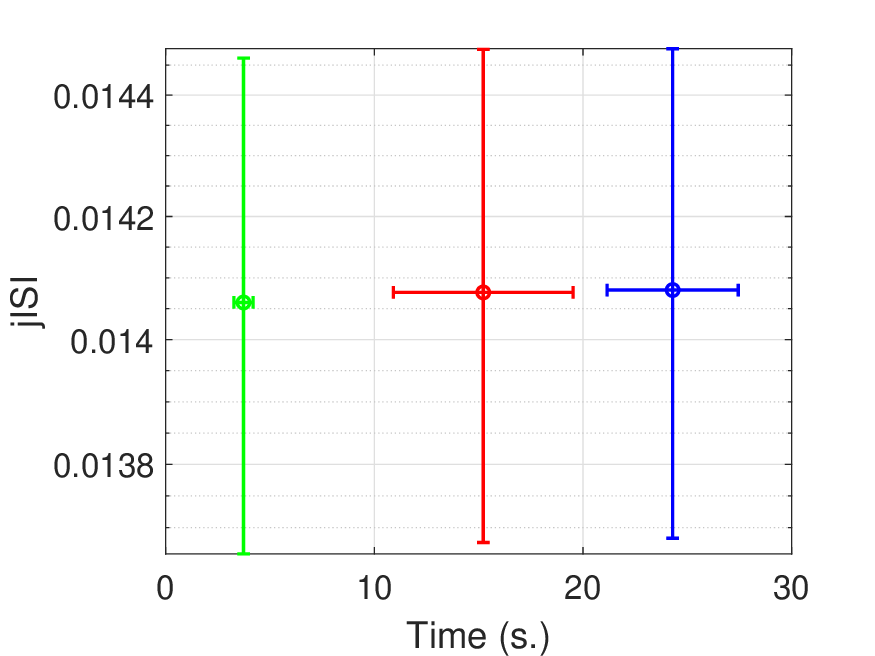}&
\includegraphics[width=3cm]{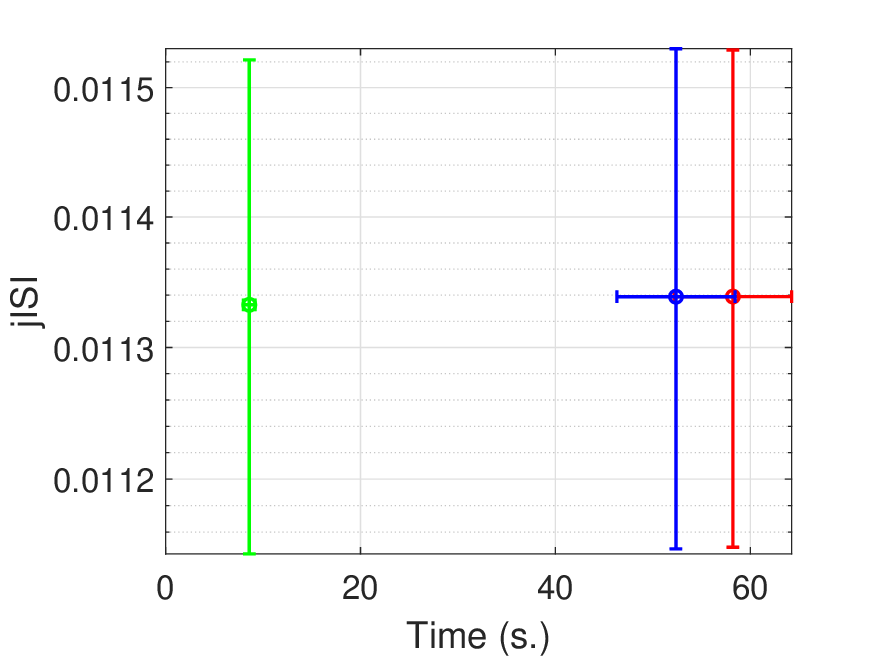}\\
$(K,N) = (5,20)$ & $(K,N) = (10,20)$ & $(K,N) = (20,20)$\\
\multicolumn{3}{c}{\textbf{Case B}: $\boldsymbol{\rho} \in [0.2,0.3]^N$ and $\lambda = 0.25$} \\
    \end{tabular}
    }
    \caption{jISI score vs computational time in seconds (mean $\pm$ standard deviation), for \textbf{PALM-IVA-G} (green), \textbf{IVA-G-V} (red) and \textbf{IVA-G-N} (blue), for Case A and Case B.}
    \label{fig:caseAtoB}
\end{figure}

\begin{figure}[ht!]
    \centering
    \scalebox{0.95}{
    \begin{tabular}{@{}c@{}c@{}c@{}}
\includegraphics[width=3cm]{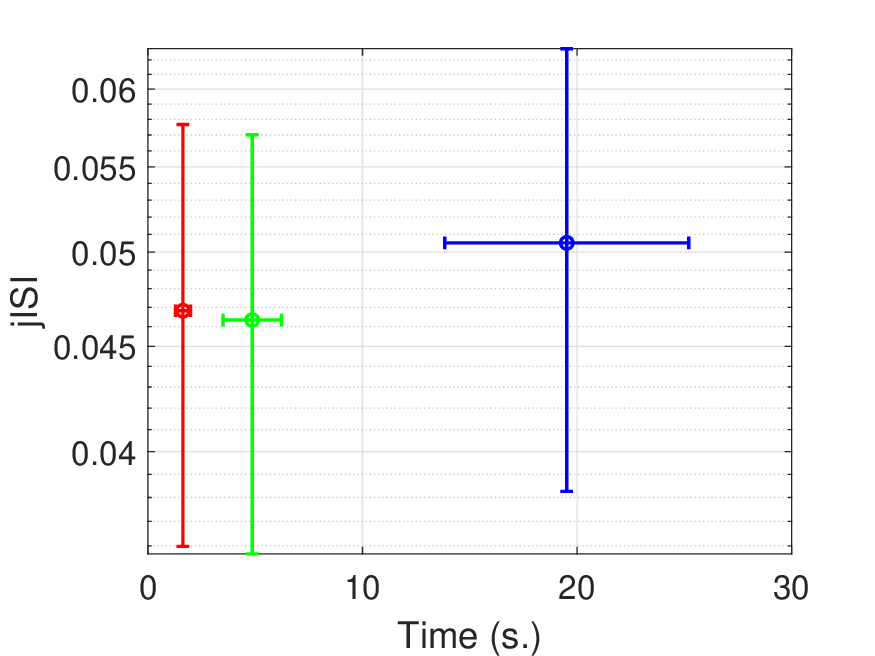}&
\includegraphics[width=3cm]{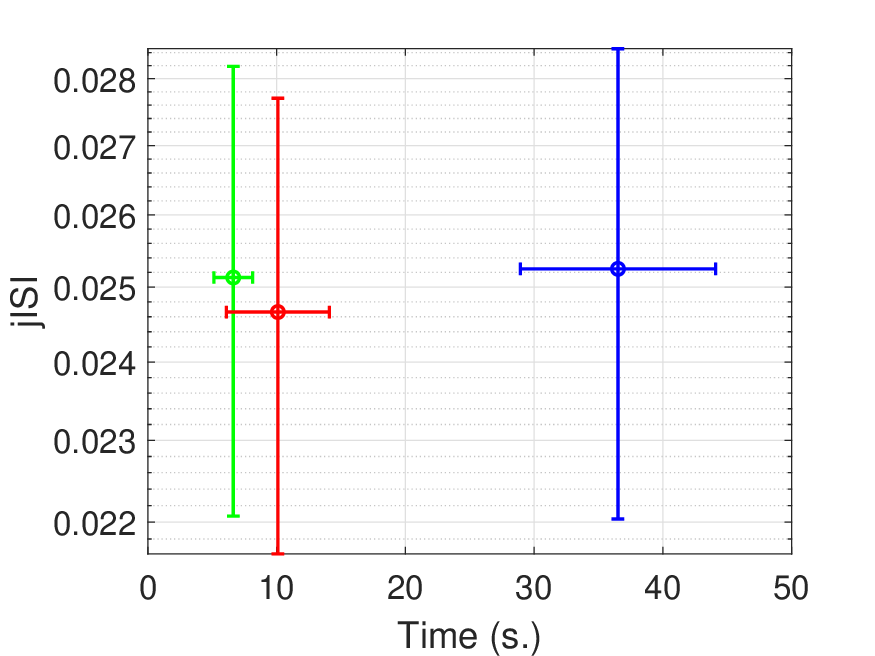}&
\includegraphics[width=3cm]{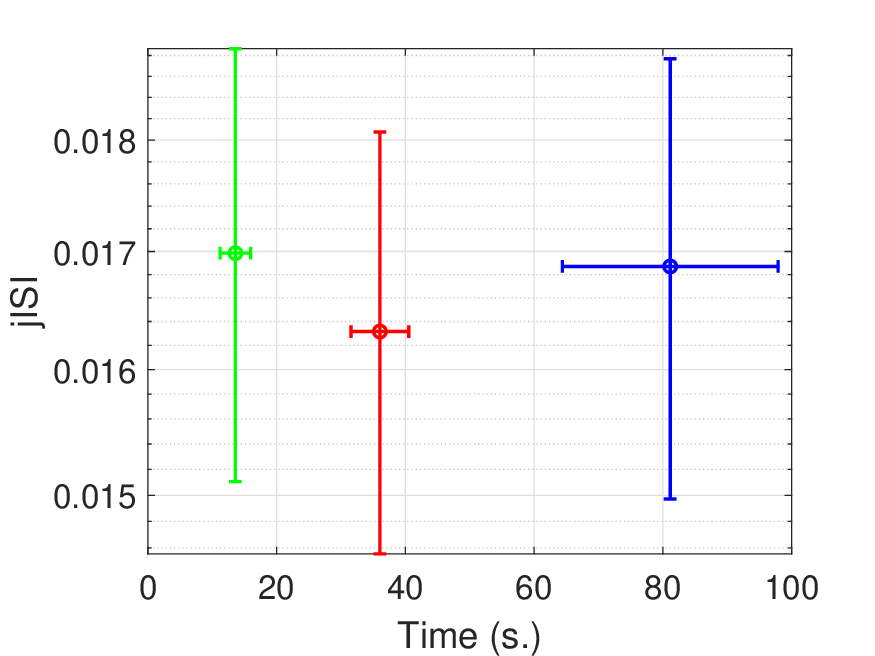}\\
$(K,N) = (5,10)$ & $(K,N) = (10,10)$ & $(K,N) = (20,10)$\\
\includegraphics[width=3cm]{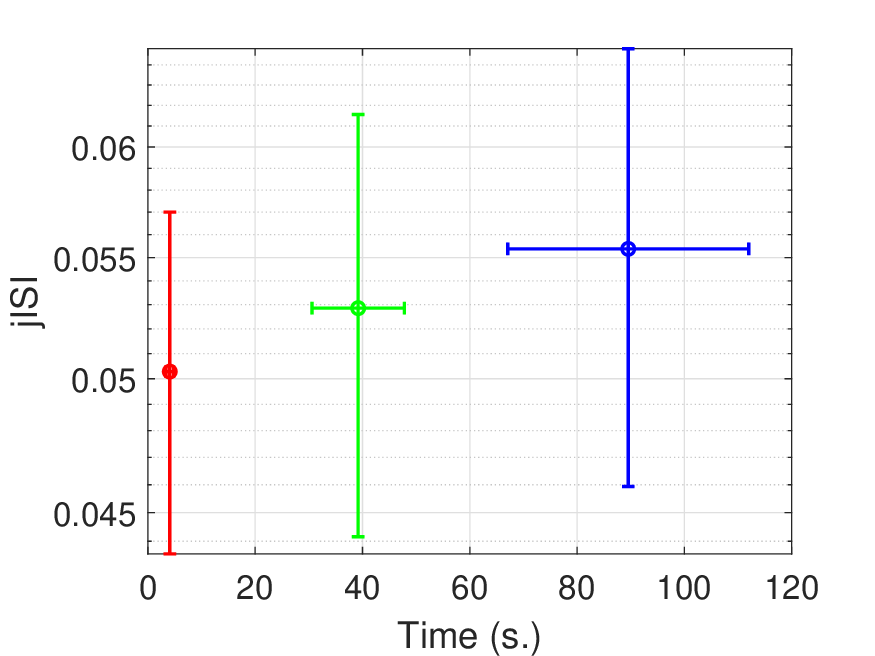}&
\includegraphics[width=3cm]{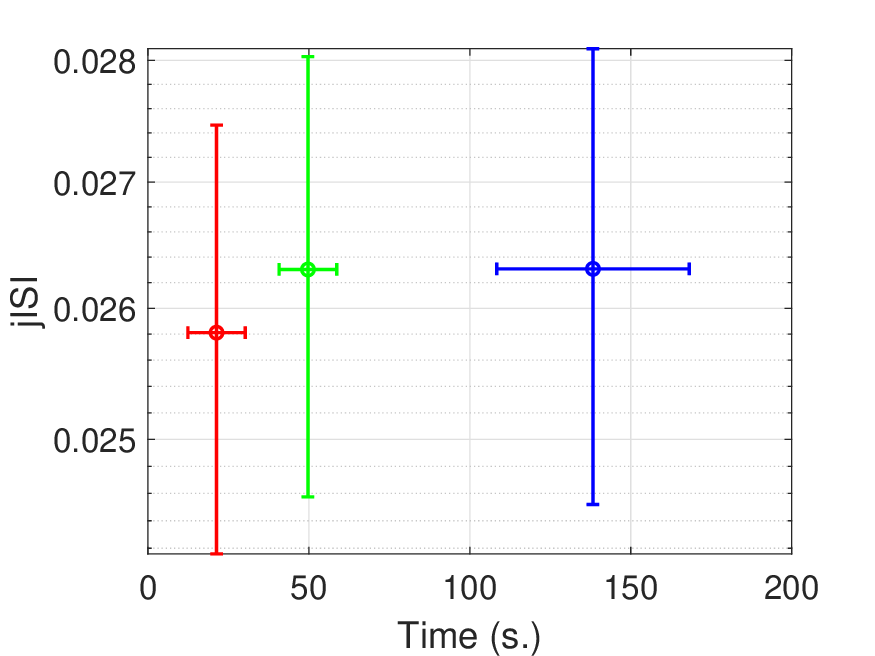}&
\includegraphics[width=3cm]{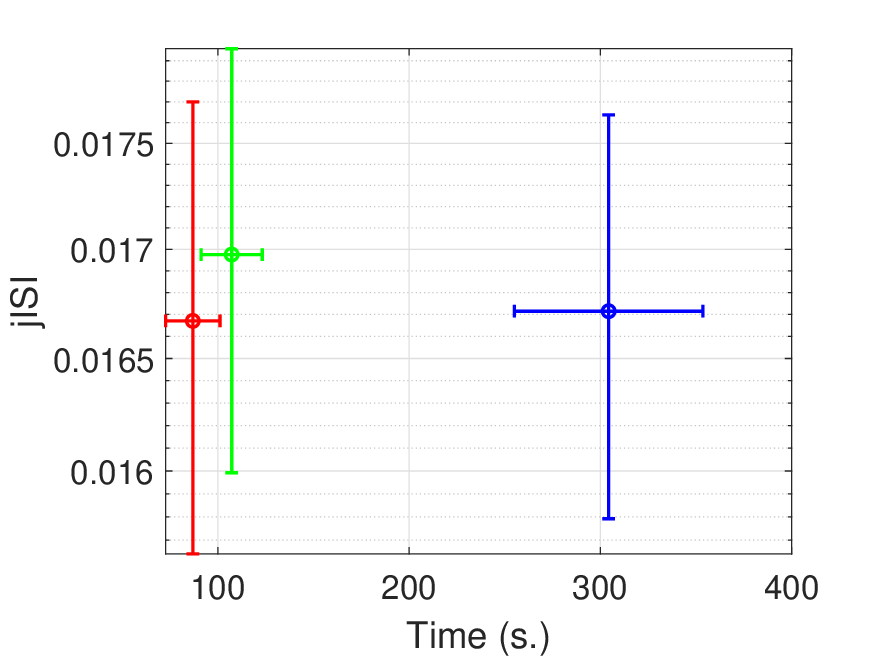}\\
$(K,N) = (5,20)$ & $(K,N) = (10,20)$ & $(K,N) = (20,20)$\\
\multicolumn{3}{c}{\textbf{Case C}:$\boldsymbol{\rho} \in [0.6,0.7]^N$ and $\lambda = 0.04$} \\
\includegraphics[width=3cm]{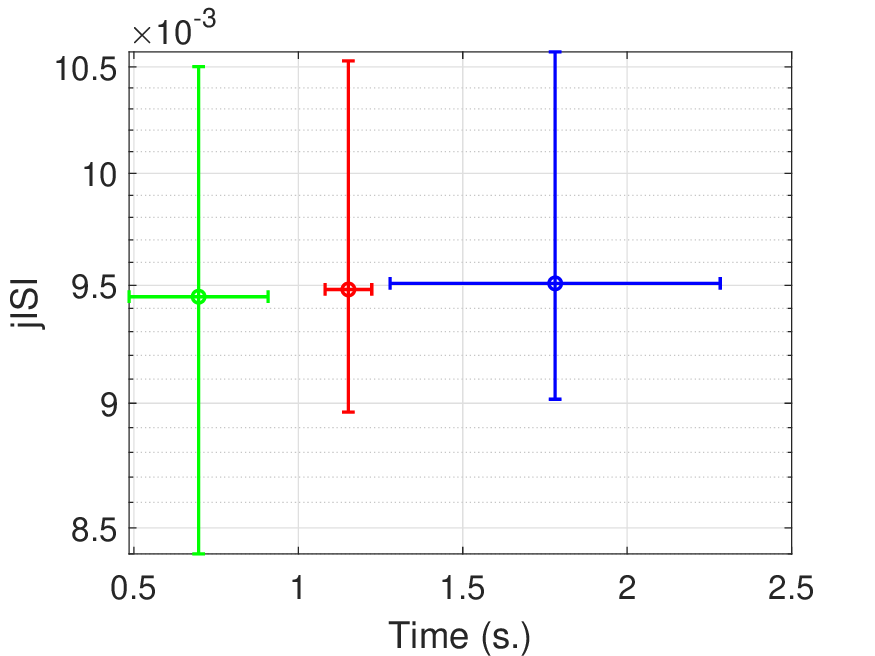}&
\includegraphics[width=3cm]{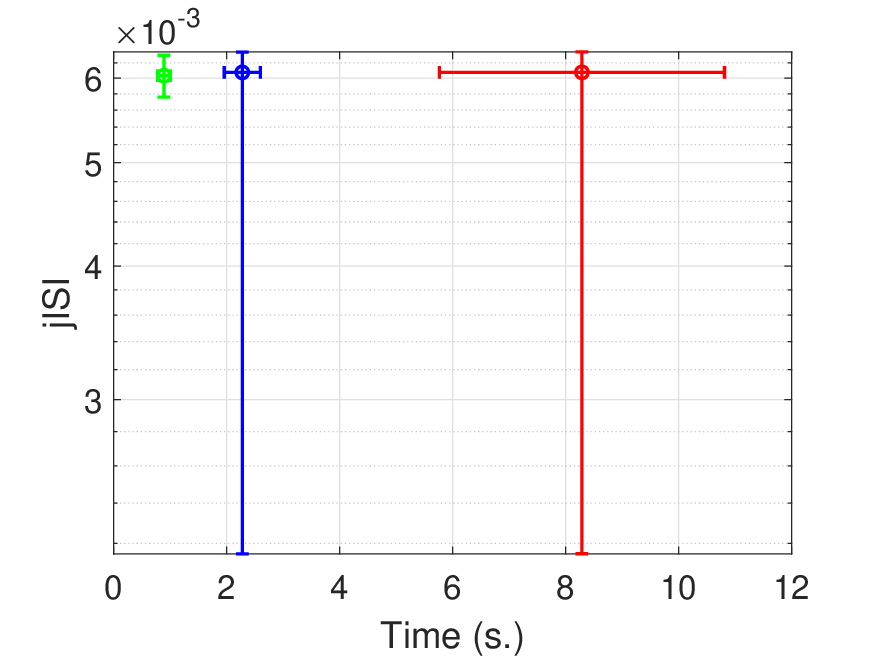}&
\includegraphics[width=3cm]{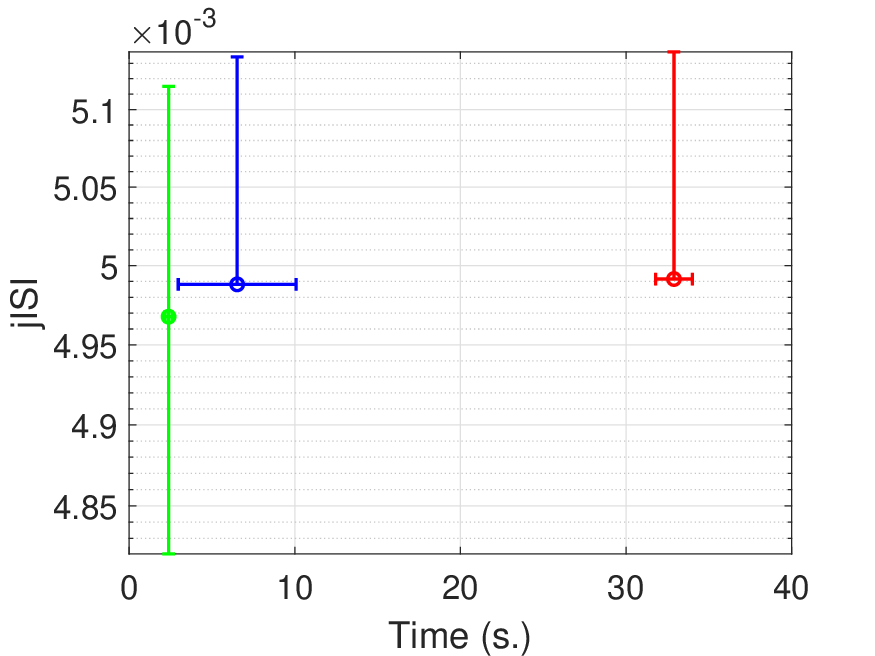}\\
$(K,N) = (5,10)$ & $(K,N) = (10,10)$ & $(K,N) = (20,10)$\\
\includegraphics[width=3cm]{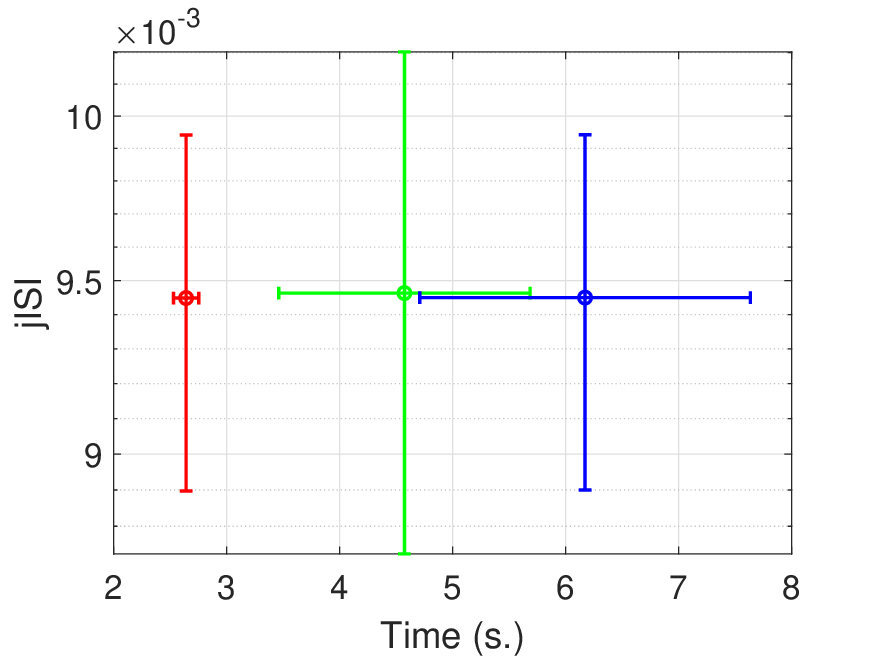}&
\includegraphics[width=3cm]{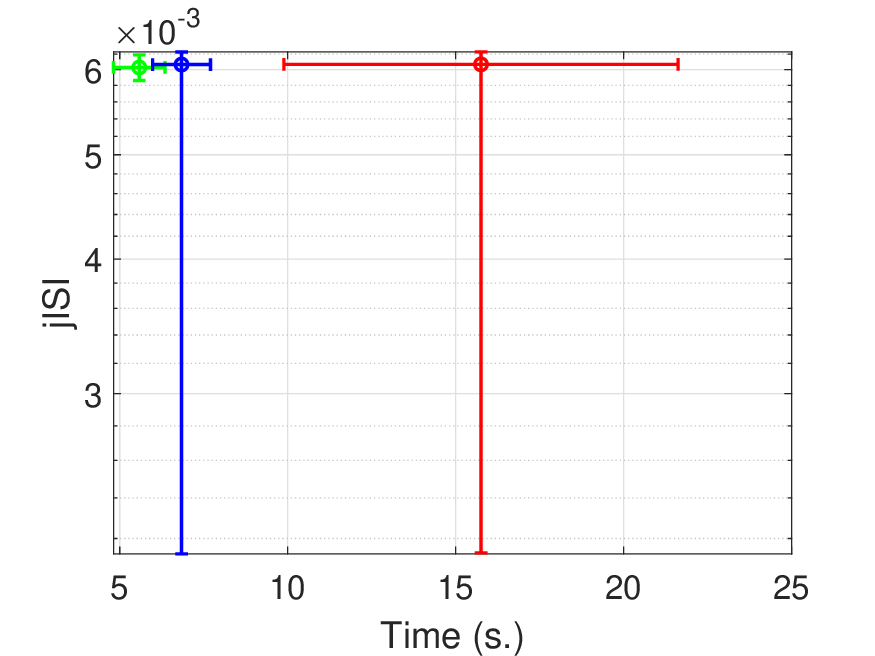}&
\includegraphics[width=3cm]{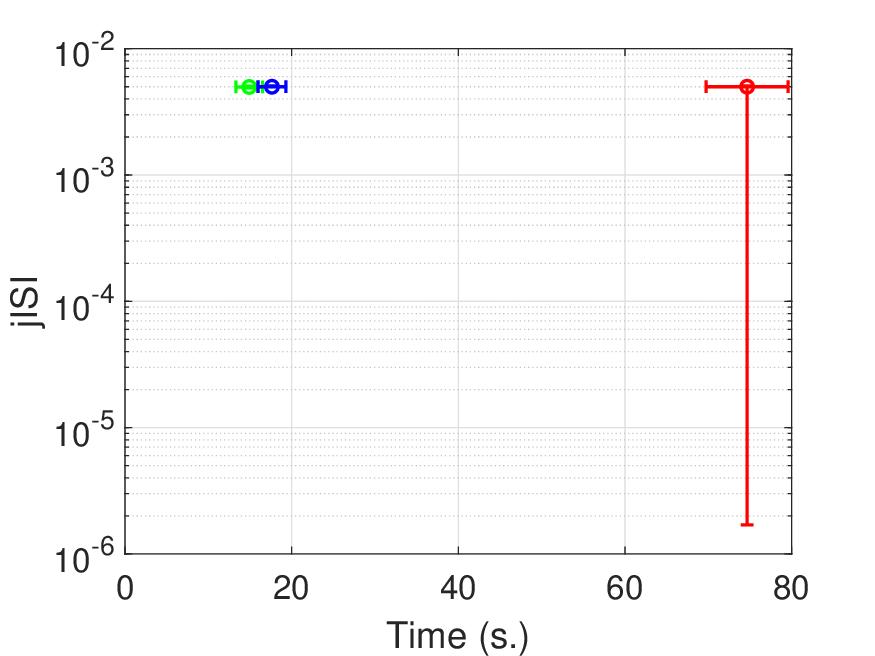}\\
$(K,N) = (5,20)$ & $(K,N) = (10,20)$ & $(K,N) = (20,20)$\\
\multicolumn{3}{c}{\textbf{Case D}:$\boldsymbol{\rho} \in [0.6,0.7]^N$ and $\lambda = 0.25$} 
    \end{tabular}
    }
    \caption{jISI score vs computational time in seconds (mean $\pm$ standard deviation), for \textbf{PALM-IVA-G} (green), \textbf{IVA-G-V} (red) and \textbf{IVA-G-N} (blue), for Cases C and D.}
    \label{fig:caseCtoD}
\end{figure}

In conclusion, in addition to its established convergence guarantees, \textbf{PALM-IVA-G} appears to be competitive with the state of the art IVA-G algorithm, consistently achieving good jISI scores and taking reasonable time to run, especially when the number of sources is not too high.

\begin{table}[ht!]
\caption{Averaged jISI scores, $\mu_{\rm jISI}$, and averaged computational times, $\mu_\texttt{T}$, in seconds, for \textbf{PALM-IVA-G}, \textbf{IVA-G-V} and \textbf{IVA-G-N} (from top to bottom). Best (i.e., lowest) jISI results (resp. lowest times) are highlighted in bold (resp. italic bold).}
\vspace{0.4cm}
\fontsize{5pt}{5pt}\selectfont
\begin{tabular}{cm{0.5cm}m{0.5cm}cccccc}
& & & \multicolumn{2}{c}{$K$ = 5} & \multicolumn{2}{c}{$K$ = 10} & \multicolumn{2}{c}{$K$ = 20}\\
 \cmidrule(lr){4-5} \cmidrule(lr){6-7} \cmidrule(lr){8-9}
& & & $N$ = 10 & $N$ = 20 & $N$ = 10 & $N$ = 20 & $N$ = 10 & $N$ = 20\\
\midrule
\multirow{12}{*}{\rotatebox[origin=c]{90}{\small{\textbf{PALM-IVA-G}}}}& \multirow{2}{*}{\begin{tabular}{c} Case A \end{tabular}}& $\mu_{\rm jISI}$ & \textbf{9.79E-02} & \textbf{1.36E-01} & \textbf{5.40E-02} & 6.74E-02 & 3.88E-02 & 4.18E-02\\
& & $\mu_\texttt{T}$ & 3.0 & 23.8 & \textit{\textbf{4.3}} & 41.2 & \textit{\textbf{9.2}} & \textit{\textbf{84.2}}\\
\cmidrule(lr){2-9}& \multirow{2}{*}{\begin{tabular}{c} Case B \end{tabular}}& $\mu_{\rm jISI}$ & \textbf{2.14E-02} & \textbf{2.20E-02} & \textbf{1.40E-02} & \textbf{1.41E-02} & \textbf{1.13E-02} & \textbf{1.13E-02}\\
& & $\mu_\texttt{T}$ & \textit{\textbf{0.5}} & 3.1 & \textit{\textbf{0.6}} & \textit{\textbf{3.7}} & \textit{\textbf{1.5}} & \textit{\textbf{8.6}}\\
\cmidrule(lr){2-9}& \multirow{2}{*}{\begin{tabular}{c} Case C \end{tabular}}& $\mu_{\rm jISI}$ & \textbf{4.63E-02} & 5.29E-02 & 2.51E-02 & 2.63E-02 & 1.70E-02 & 1.70E-02\\
& & $\mu_\texttt{T}$ & 4.9 & 39.2 & \textit{\textbf{6.6}} & 49.7 & \textit{\textbf{13.6}} & 107.2\\
\cmidrule(lr){2-9}& \multirow{2}{*}{\begin{tabular}{c} Case D \end{tabular}}& $\mu_{\rm jISI}$ & \textbf{9.45E-03} & \textbf{9.46E-03} & \textbf{6.03E-03} & \textbf{6.03E-03} & \textbf{4.97E-03} & \textbf{4.98E-03}\\
& & $\mu_\texttt{T}$ & \textit{\textbf{0.7}} & 4.6 & \textit{\textbf{0.9}} & \textit{\textbf{5.6}} & \textit{\textbf{2.4}} & \textit{\textbf{14.9}}\\
\bottomrule
\\
\midrule
\multirow{12}{*}{\rotatebox[origin=c]{90}{\small{\textbf{IVA-G-V}}}}& \multirow{2}{*}{\begin{tabular}{c} Case A \end{tabular}}& $\mu_{\rm jISI}$ & 1.01E-01 & 1.37E-01 & 5.61E-02 & \textbf{6.64E-02} & \textbf{3.81E-02} & \textbf{4.11E-02}\\
& & $\mu_T$ & \textit{\textbf{1.8}} & \textit{\textbf{5.4}} & 12.2 & \textit{\textbf{30.4}} & 44.6 & 123.7\\
\cmidrule(lr){2-9}& \multirow{2}{*}{\begin{tabular}{c} Case B \end{tabular}}& $\mu_{\rm jISI}$ & \textbf{2.15E-02} & \textbf{2.20E-02} & \textbf{1.40E-02} & \textbf{1.41E-02} & \textbf{1.13E-02} & \textbf{1.13E-02}\\
& & $\mu_\texttt{T}$ & 1.1 & \textit{\textbf{2.3}} & 6.0 & 15.2 & 28.5 & 58.2\\
\cmidrule(lr){2-9}& \multirow{2}{*}{\begin{tabular}{c} Case C \end{tabular}}& $\mu_{\rm jISI}$ & 4.68E-02 & \textbf{5.03E-02} & \textbf{2.47E-02} & \textbf{2.58E-02} & \textbf{1.63E-02} & \textbf{1.67E-02}\\
& & $\mu_\texttt{T}$ & \textit{\textbf{1.6}} & \textit{\textbf{4.1}} & 10.1 & \textit{\textbf{21.3}} & 36.0 & \textit{\textbf{86.8}}\\
\cmidrule(lr){2-9}& \multirow{2}{*}{\begin{tabular}{c} Case D \end{tabular}}& $\mu_{\rm jISI}$ & \textbf{9.48E-03} & \textbf{9.45E-03} & \textbf{6.08E-03} & \textbf{6.07E-03} & \textbf{4.99E-03} & \textbf{5.00E-03}\\
& & $\mu_\texttt{T}$ & 1.2 & \textit{\textbf{2.6}} & 8.3 & 15.8 & 32.9 & 74.7\\
\bottomrule
\\
\midrule
\multirow{12}{*}{\rotatebox[origin=c]{90}{\small{\textbf{IVA-G-N}}}}& \multirow{2}{*}{\begin{tabular}{c} Case A \end{tabular}}& $\mu_{\rm jISI}$ & 1.36E-01 & 1.69E-01 & 7.62E-02 & 9.79E-02 & 4.76E-02 & 4.95E-02\\
& & $\mu_\texttt{T}$ & 25.0 & 82.8 & 76.4 & 285.7 & 234.7 & 1023.5\\
\cmidrule(lr){2-9}& \multirow{2}{*}{\begin{tabular}{c} Case B \end{tabular}}& $\mu_{\rm jISI}$ & 2.19E-02 & 2.21E-02 & \textbf{1.41E-02} & \textbf{1.41E-02} & \textbf{1.13E-02} & \textbf{1.13E-02}\\
& & $\mu_\texttt{T}$ & 5.9 & 21.9 & 7.7 & 24.3 & 17.0 & 52.4\\
\cmidrule(lr){2-9}& \multirow{2}{*}{\begin{tabular}{c} Case C \end{tabular}}& $\mu_{\rm jISI}$ & 5.05E-02 & 5.54E-02 & 2.52E-02 & 2.63E-02 & 1.69E-02 & \textbf{1.67E-02}\\
& & $\mu_\texttt{T}$ & 19.5 & 89.5 & 36.5 & 138.3 & 81.1 & 304.3\\
\cmidrule(lr){2-9}& \multirow{2}{*}{\begin{tabular}{c} Case D \end{tabular}}& $\mu_{\rm jISI}$ & \textbf{9.51E-03} & \textbf{9.45E-03} & \textbf{6.08E-03} & \textbf{6.06E-03} & \textbf{4.99E-03} & \textbf{5.00E-03}\\
& & $\mu_\texttt{T}$ & 1.8 & 6.2 & 2.3 & 6.8 & 6.5 & 17.6\\
\bottomrule
\\
\end{tabular}
\label{tab:caseAtoD}
\end{table}

\section{Conclusion}
\label{sec:conc}

In this article, we addressed the problem of joint blind source separation, through the IVA-G formulation. First, we derived a cost function parameterised by the demixing matrices and the precision matrices of the sources, to solve IVA-G based on the maximum likelihood estimator. We then introduced an additional term to fix the scaling ambiguity, hence forcing the precision matrices of output solution to have ones on its diagonal, completing the design of our cost function.

Then, we studied the different terms of this non-convex and non-smooth cost, in particular, we gave explicit formulas for the partial gradients of the smooth term with respect to both blocks of variable and for the proximity operators of the separable terms. Based on these results, we proposed an algorithm adapted from the $\textbf{PALM}$ optimizer, and proved that the sequence of iterates converges globally to a critical point. We compared our method to two state-of-the-art IVA-G algorithms and showed that our method is competitive in terms of jISI score and computation time, especially for a moderate number of sources.

As a future work, those encouraging results on synthetic datasets should be confirmed on real data, for instance from fMRI or MEEG. Depending on the application considered, we could leverage model knowledge in the form of new regularization terms in the cost function. The proposed proximal alternating algorithm is versatile and would be easily adaptable to a large class of penalties and constraints.

\begin{appendices}

\section{Proof of Lemma \ref{lem:gradW}}
\label{proof gradW}
Let $(\W,\C) \in \R^{N \times N \times K} \times \R^{K \times K \times N}$. The partial derivative of $h$ in \eqref{eq:hpalm}, with respect to the $(n,m,k)$-th entry of tensor $\W$ reads
\begin{equation}
\frac{\partial h(\W,\C)}{\partial w_{n,m}^{[k]}} = \frac{1}{2}\sum_{n'=1}^N \frac{\partial \tr (\bC_{n'}\bW_{n'}\Rx \bW_{n'}^\top)}{\partial w_{n,m}^{[k]}}.
\end{equation}

For $n' \neq n$, the terms of the above sum are null, hence
\begin{equation}
\frac{\partial h(\W,\C)}{\partial w_{n,m}^{[k]}} = \frac{1}{2} \frac{\partial \tr (\bC_n\bW_n\Rx \bW_n^\top)}{\partial w_{n,m}^{[k]}}.
\end{equation}
Then, by applying the formula of the derivative of a matricial scalar product \cite{Petersen2008}, 

\begin{align}
    & \frac{\partial h(\W,\C)}{\partial w_{n,m}^{[k]}} = \frac{1}{2} \frac{\partial \langle \bC_n \bW_n\Rx|\bW_n \rangle}{\partial w_{n,m}^{[k]}} \nonumber\\
    &\; = \frac{1}{2} \Big( \langle \frac{\partial (\bC_n \bW_n\Rx)}{\partial w_{n,m}^{[k]}}|\bW_n\rangle + \langle \bC_n \bW_n\Rx|\frac{\partial \bW_n}{\partial w_{n,m}^{[k]}} \rangle \Big) \nonumber\\
    &\; = \frac{1}{2} \Big( \langle \bC_n\frac{\partial \bW_n}{\partial w_{n,m}^{[k]}}\Rx|\bW_n\rangle + \langle \bC_n \bW_n\Rx|\frac{\partial \bW_n}{\partial w_{n,m}^{[k]}} \rangle \Big) \nonumber\\
    &\; = \tr \Big(\frac{\partial \bW_n}{\partial w_{n,m}^{[k]}}\Rx \bW_n^\top \bC_n\Big).
\end{align}
Hereabove, $\frac{\partial \bW_n}{\partial w_{n,m}^{[k]}}$ is a matrix of dimension $K\times KN$ whose elements are equal to $0$, except one, at row index $k$ and column index $(k-1)N+m$, which is equal to $1$. We deduce
\begin{equation}
        \frac{\partial h(\W,\C)}{\partial w_{n,m}^{[k]}}
        = [\bC_n\bW_n\Rx]_{k,(k-1)N+m}.
\end{equation}
The above expression can be reexpressed in a more concise matrix form which gives \eqref{grad_W J}, and proves the first part of the 
lemma.  

Let $\C \in \R^{K \times K \times N}$. From \eqref{grad_W J}, we can see that $\nabla_{\W} h(.,\C)$ is linear, and thus Lipschitz continuous. Moreover, for every $\W \in \R^{N \times N \times K}$,
\begin{align}
\Big\|\frac{\partial h(\W,\C)}{\partial \W}\Big\|^2
    &= \sum_{k=1}^K \sum_{n=1}^N \sum_{m=1}^N ([\bC_n\bW_n\Rx]_{k,(k-1)N+m})^2 \nonumber\\
    &= \sum_{n=1}^N \sum_{k=1}^K \big( \sum_{m=1}^N  ([\bC_n\bW_n\Rx]_{k,(k-1)N+m})^2 \big) \nonumber\\
    &= \sum_{n=1}^N \sum_{k=1}^K \|[\bC_n\bW_n\Rx]_{k,(k-1)N+1,\ldots,N}\|^2 \nonumber\\
    &= \sum_{n=1}^N \sum_{k=1}^K \|\bc_{n,k}\bW_n\Rx^{[k]}\|^2. 
\end{align}

Using the identity $\|\bA \sbmm{B}\| \leq \|\bA\|_{\rm S} \|\sbmm{B}\|$,
 we have, for all $k\in \kfirst, n \in \nfirst$,
\begin{equation}    \|\bc_{n,k}\bW_n\Rx^{[k]}\|^2 \leq \|\Rx^{[k]}\|_{\rm S}^2 \|\bc_{n,k}\bW_n\|^2 \leq \varrho_{\Rx}^2 \|\bc_{n,k}\bW_n\|^2.
\end{equation}
$\bC_n \bW_n$ is a $K \times KN$ matrix whose $k$-th line is equal to $\bc_{n,k}\bW_n$, hence,
\begin{equation}
\sum_{k=1}^K \|\bc_{n,k}\bW_n\|^2 =  \|\bC_n \bW_n\|^2 \leq \|\bC_n\|_{\rm S}^2 \|\bW_n\|^2.
\end{equation}
Overall, 
\begin{align}
    \Big\|\frac{\partial h(\W,\C)}{\partial \W}\Big\|^2 &\leq  \sum_{n=1}^N \varrho_{\Rx}^2 \sum_{k=1}^K \|\bc_{n,k}\bW_n\|^2 \nonumber\\
    &\leq \varrho_{\Rx}^2 \sum_{n=1}^N \rho_{\C}^2 \|\bW_n\|^2 \nonumber\\
    &= \rho_{\C}^2 \varrho_{\Rx}^2 \|\W\|^2.
\end{align}
This concludes the second part of the Lemma, i.e. $\W \to h(\W,\C)$ is Lipschitz differentiable with constant \eqref{eq:lipW}.


\section{Proof of Lemma \ref{lem:gradC}}
\label{proof gradC}
For all $\W \in \R^{N \times N \times K}$, $h(\W,.)$ in \eqref{eq:hpalm} is quadratic, which yields the expression  \eqref{grad_C J} in a straightforward manner, and concludes the first part of the proof. Let $\W' \in \R^{N \times N \times K}$. According to \eqref{grad_C J}, it can be easily checked that $\nabla_{\C}h(\W',.)$ is Lipschitz with moduli $\alpha$, hence $\C \to h(\W',\C)$ is Lipschitz differentiable with constant \eqref{eq:lipC}, which ends the proof.

\section{proof of Lemma \ref{proxW}}
\label{proof proxW}
Function $f$ in \eqref{eq:fPALM}, reads, equivalently, for all $ \W \in \R^{N \times N \times K}$, $f(\W) = \sum_{k=1}^K \hat{f}(\bW^{[k]})$, where, for all $\sbmm{M} \in \mathbb{R}^{N \times N}$, with singular values $\boldsymbol{\sigma}_{\sbmm{M}}= (\sigma_{\sbmm{M},\ell})_{1 \leq \ell \leq N}$, $\hat{f}(\sbmm{M}) = - \log |\det \sbmm{M}| = \phi_{\hat{f}}(\boldsymbol{\sigma}_{\sbmm{M}})$,
with
\begin{align}
    \phi_{\hat{f}} \colon \R^N &\mapsto (-\infty,+\infty]\nonumber \\ \boldsymbol{\sigma} = (\sigma_1,\ldots,\sigma_N) &\mapsto
    \begin{cases}
        +\infty  ~ \text{if} ~ (\exists \ell \in \{1,\ldots,N\})\; \sigma_{\ell} \leq 0\\
    -\sum_{l=1}^N \log \sigma_{\ell} ~ \text{otherwise}.
    \end{cases}
    \label{eq:phihatf}
\end{align}

Let $c > 0$ and $ \W' \in \R^{N \times N \times K}$. Then,
\begin{equation*}
     cf(\W) + \frac{1}{2} \|\W - \W'\|^2 = \sum_{k=1}^K c\hat{f}(\bW^{[k]}) + \frac{1}{2} \|\bW^{[k]} - \bW^{'[k]}\|^2.
\end{equation*}
Hence, we can minimize the sum by minimizing all its terms independently. In other words, $\W \in \prox_{cf}(\W')$, if and only if, for every $k \in \{1,\ldots,K\}$,
\begin{equation*}
\bW^{[k]} \in \prox_{c\hat{f}}(\bW^{'[k]}).
\end{equation*}
It remains to derive an explicit form for $\prox_{c\hat{f}}$. To do so, let us use the following lemma whose proof is given in \cite{Pesquet} (see also \cite{benfenati_proximal_2018,Bauschke2017} for similar results on the proximity operators of spectral functions).

\begin{lemma}
\label{prox_mat}
Let $N \geq 1$ and $\phi : \R^N \mapsto (-\infty,+\infty]$ be a function whose domain is not empty and is contained in $[0,+\infty)^N$, such that $\phi$ is invariant by any permutation of its arguments and that the proximity operator of $\phi$ is defined.\\
Let $\Phi : \R^{N \times N} \mapsto  (-\infty,+\infty], \sbmm{M} \mapsto \phi(\boldsymbol{\sigma}_{\sbmm{M}})$ where $\boldsymbol{\sigma}_{\sbmm{M}}$ is a vector containing the $N$ singular values of $\mathbf{M}$ in any order.\\
Then, the proximity operator of $\Phi$ is defined, and for any $\sbmm{M}' \in \R^{N \times N}$, whose singular value decomposition reads $\sbmm{M}' = \sbmm{U}_{\sbmm{M}'} \Diag(\boldsymbol{\sigma}_{\sbmm{M}'}) \sbmm{V}_{\sbmm{M}'}^\top$, if $\boldsymbol{\sigma}_{\sbmm{M}} \in \prox_{\phi}(\boldsymbol{\sigma}_{\sbmm{M}'})$, then we have $\sbmm{U}_{\sbmm{M}'} \Diag(\boldsymbol{\sigma}_{\sbmm{M}}) \sbmm{V}_{\sbmm{M}'}^\top \in \prox_{\Phi}(\sbmm{M}')$. 
\end{lemma}

Function \eqref{eq:phihatf} is invariant by permutation of its arguments and verifies $\varnothing \neq \dom \phi_{\hat{f}} \subset [0,+\infty)^N$. Moreover, for any $\boldsymbol{\sigma}' = (\sigma_1',\ldots,\sigma_N') \in \R^N$, 
\begin{align}   
    \prox_{c\phi_{\hat{f}}}(\boldsymbol{\sigma}') = ~ &\underset{\boldsymbol{\sigma} \in \R^N}{\argmin} ~ cf(\boldsymbol{\sigma}) + \frac{1}{2} \|\boldsymbol{\sigma} - \boldsymbol{\sigma}'\|^2\nonumber \\ = &\underset{\boldsymbol{\sigma} \in (0,+\infty)^N}{\argmin} \sum_{\ell=1}^N \frac{1}{2}(\sigma_{\ell} - \sigma_{\ell}')^2 - c \log \sigma_{\ell}.
\end{align}
The above prox calculation requires to minimize a sum of functions, each term acting on a different variable, which means we can minimize each term independently. Let $l \in \{1,\ldots,N\}$, then
\begin{align*}
    \forall \sigma_\ell > 0 &: \frac{d \big(\frac{1}{2}(\sigma_\ell - \sigma_\ell')^2 - c \log \sigma_\ell\big)}{d\sigma_\ell} = \sigma_\ell - \sigma_\ell' - c \sigma_\ell^{-1} = 0\\
    \iff &\sigma_\ell^2 - \sigma_\ell' \sigma_\ell - c = 0\\
    \iff &\sigma_\ell \in \Bigg\{ \frac{\sigma_\ell' - \sqrt{(\sigma_\ell')^2 + 4c}}{2}, \frac{\sigma_\ell' + \sqrt{(\sigma_\ell')^2 + 4c}}{2} \Bigg\}.
\end{align*}

As $\sigma_\ell \mapsto \frac{1}{2}(\sigma_\ell - \sigma_\ell')^2 - c\log \sigma_\ell$ diverges toward $+\infty$ in $0^+$ and $+\infty$ and is $C^1$ on $(0,+\infty)$, it must have a minimum where its derivative is equal to $0$. The only point where it happens is $\sigma_\ell = \frac{\sigma_\ell' + \sqrt{(\sigma_\ell')^2 + 4c}}{2}$, so we can conclude that this point is a global minimum.\\

Finally, $\prox_{c\phi_{\hat{f}}}$ is uniquely defined for any $\boldsymbol{\sigma}' \in \R^N$ and has the explicit form:
$$\boldsymbol{\sigma} = \prox_{c\phi_{\hat{f}}}(\boldsymbol{\sigma}') = \frac{\boldsymbol{\sigma}' + \sqrt{(\boldsymbol{\sigma}')^2 + 4c}}{2}$$ where the square and square root operations are applied component-wise. Applying Lemma \ref{prox_mat}, with $\sbmm{M} = \sbmm{W}^{[k]}$, $\sbmm{M}' = \bW^{'[k]}$, for every $k \in \{1,\ldots,K\}$, and $\Phi = \widehat{f}$, allows to conclude the proof.

\section{proof of Lemma \ref{proxC}}
\label{proof proxC}
For every $\C \in \R^{K \times K \times N}$, we can rewrite function \eqref{eq:gPALM}, as $g(\C) = \sum_{n=1}^N \widehat{g} (\bC_n)$ where, for every $\sbmm{M} \in \R^{K \times K}$, with singulare values $\boldsymbol{\sigma}_{\sbmm{M}} =  (\sigma_{\sbmm{M},\ell})_{1 \leq \ell \leq K}$, $\widehat{g}(\sbmm{M}) =  \phi_{\hat{g}}(\boldsymbol{\sigma}_{\sbmm{M}})$, 
with
\begin{align*}
    \phi_{\hat{g}} : \R^K &\to (-\infty,+\infty]\\
    \boldsymbol{\sigma} = (\sigma_1,\ldots,\sigma_K) &\mapsto 
        \begin{cases}
        +\infty  ~ \text{if} ~ (\exists \ell \in \{1,\ldots,K\})\; \sigma_{\ell} \leq 0\\
        -\sum_{l=1}^K \log \sigma_{\ell} ~ \text{otherwise}.
    \end{cases}    
\end{align*}
Moreover, for any $\boldsymbol{\sigma}' \in \R^K$,
\begin{align*}
    \prox_{c\phi_{\hat{g}}}(\boldsymbol{\sigma}') &= \underset{\boldsymbol{\sigma} \in [\epsilon,+\infty)}{\argmin} \frac{1}{2}\sum_{\ell=1}^K (\sigma_\ell - \sigma_\ell')^2 - c\log \sigma_\ell\\
    &= \max \Bigg( \epsilon, \frac{\boldsymbol{\sigma}' + \sqrt{(\boldsymbol{\sigma}')^2+2c}}{2} \Bigg).
\end{align*}
Hence the result, applying Lemma \ref{prox_mat} for every $n \in \{1,\ldots,N\}$, with $\Phi = \hat{g}$, and $(\sbmm{M},\sbmm{M}') = (\sbmm{C}_n,\sbmm{C}_n')$.

\section{Convergence proof}
\label{covergence proof}

\paragraph{Lower-boundedness of the cost function}

Using Theorem \ref{equivalence ML MI},
\begin{multline*}
(\forall (\W,\C) \in \R^{N \times N \times K} \times \R^{K \times K \times N}) \\ J_{\rm IVA-G}(\W,\C) \geq \widetilde{J}_{\rm IVA-G}(\W) + \frac{KN}{2}.
\end{multline*}
 Let us note $\sigma^-$ the smallest eigenvalue of $\Rx$, then $\Rx - \sigma^- \sbmm{I}_{KN}$ is a symmetric matrix whose eigenvalue are non-negative, i.e., $\Rx - \sigma^- \sbmm{I}_{KN}$ is symmetric positive. Therefore, it is also the case of $\bW_n(\Rx - \sigma^- \sbmm{I}_{KN})\bW_n^\top$ for any $n \in \nfirst$, and
 \begin{equation}
 \label{WRW>lambda||W||}
      \bW_n\Rx \bW_n^\top \succeq \sigma^- \bW_n\bW_n^\top, 
 \end{equation}
 where $\succeq$ denotes the Loewner order relationship defined on $(\mathcal{S}_N)^2$ as $(\forall (\sbmm{A}, \sbmm{B}) \in (\mathcal{S}_N)^2) ~ \sbmm{A} \succeq \sbmm{B} \iff \sbmm{A} - \sbmm{B} \in \mathcal{S}_N^+$.\\
 
 Yet, $\bW_n\bW_n^\top$ is a diagonal matrix whose $k$-th coefficient is $\|\bw_n^{[k]}\|^2$. Hence $\det(\bW_n\Rx \bW_n^\top) \geq (\sigma^-)^K\det(\bW_n\bW_n^\top)$ and
\begin{equation}
\label{minoration term 1}
\frac{1}{2} \log\det(\bW_n\Rx \bW_n^\top) \geq \frac{1}{2} K\log(\sigma^-) + \frac{1}{2}\sum_{k=1}^K \log \|\bw_n^{[k]}\|^2.
\end{equation}
Besides, by Hadamard inequality applied to the lines of $\bW^{[k]}$ for any $k \in \kfirst,\, |\det \bW^{[k]}|^2 \leq \prod_{n=1}^N \|\bw_n^{[k]}\|^2$, hence,
\begin{equation}
\label{majoration term 2}
- \log |\det \bW^{[k]}| \geq - \frac{1}{2}\sum_{n=1}^N \log \|\bw_n^{[k]}\|^2.
\end{equation}

By summing \eqref{minoration term 1} for $n \in \nfirst$, and \eqref{majoration term 2} for $k \in \kfirst$, it comes:
\begin{equation}
(\forall \W  \in \R^{N \times N \times K} )\, 
\label{lower bound for Jiva}
\widetilde{J}_{\rm IVA-G}(\W) \geq \frac{KN}{2} (1 + \log(\sigma^-)).
\end{equation}

We can conclude that $J_{\rm IVA-G}$ is lower bounded on $\R^{N \times N \times K} \times \R^{K \times K \times N}$. Finally, since the quadratic regularization term is positive, $J_{\rm IVA-G}^{\text{Reg}}$ is lower bounded too, which ends the proof.

\paragraph{Lipschitz continuity of the gradient}

The expressions of the partial gradients of $h$ given in Lemmas \ref{lem:gradW} and \ref{lem:gradC} show clearly that $\nabla h$ is a $C^1$ function on $\R^{N \times N \times K} \times \R^{K \times K \times N}$. Consequently, the mean value theorem can be applied to show that it has Lipschitz continuity on any bounded domain.

\paragraph{Boundedness of the sequence}

$\phantom{a}$\\
$\bullet$ {Boundedness of \texorpdfstring{$(\C^{(i)})_{i \in \N}$}{Ci}}\\

First, we notice that \eqref{rho_C} defines a norm on $\R^{K \times K \times N}$. For all $(i,j) \in \N^2$ such that $\C^{(i,j)}$ belongs to the sequance generated by $\textbf{PALM-IVA-G}$, let us note:
\begin{align*}
    \overline{\bC}_n^{(i,j)} = \bC_n^{(i,j)} &- c_{\C}(\alpha(\Diag( \bC_n^{(i,j)}) - \sbmm{I}_K)\\
    &- \frac{1}{2} \bW_n^{(i+1)}\Rx\bW_n^{(i+1)\top}).
\end{align*}

\begin{align}
\label{majoration C_n i+1}
\|\bC_n^{(i,j+1)}\|_{\rm S} &= \max \Big( \epsilon,\frac{\|\overline{\bC}_n^{(i,j)}\|_{\rm S} + \sqrt{\|\overline{\bC}_n^{(i,j)}\|_{\rm S}^2+2c_{\C}}}{2} \Big) \nonumber \\
&\leq \|\overline{\bC}_n^{(i,j)}\|_{\rm S} + \sqrt{\frac{c_\C}{2}},
\end{align}
assuming that $\epsilon < \sqrt{\frac{c_\C}{2}}$, which is verified in our experimental settings.

As $\bW_n^{(i+1)}\Rx\bW_n^{(i+1)\top}$ is symmetric positive,
\begin{align}
\label{majoration C_n i bar}
    \nonumber \|\overline{\bC_n}^{(i,j)}\|_{\rm S} &\leq \| \bC_n^{(i,j)} - \gamma_\C (\Diag( \bC_n^{(i,j)}) - \sbmm{I}_K)\|_{\rm S}\\
    &= \| \bC_n^{(i,j)} - \gamma_\C  \Diag( \bC_n^{(i,j)})\|_{\rm S} + \gamma_\C.
\end{align}

Let $\bu \in \R^K$ be a unit norm vector. Using the inequality for matrices of $\S^+_K, \forall (k,l) \in \{1,\ldots,K\}\, |c_{n,k,l}^{(i,j)}| \leq \sqrt{c_{n,k,k}^{(i,j)}}\sqrt{c_{n,l,l}^{(i,j)}}$, it comes
\begin{align*}
    &\bu^\top  \bC_n^{(i,j)} \bu\\
    = &\sum_{1 \leq k,l \leq K} c_{n,k,l}^{(i,j)} u_k u_l \leq \sum_{1 \leq k,l \leq K} |c_{n,k,l}^{(i,j)}| |u_k| |u_l|\\
    \leq &\sum_{1 \leq k,l \leq K} \sqrt{c_{n,k,k}^{(i,j)}} |u_k| \sqrt{c_{n,l,l}^{(i,j)}} |u_l|\\
    \leq &\sum_{1 \leq k,l \leq K} \frac{c_{n,k,k}^{(i,j)} u_k^2}{2} + \frac{c_{n,l,l}^{(i,j)} u_l^2}{2} = K \bu^\top \Diag(\bC_n^{(i,j)}) \bu.
\end{align*}
It follows that, for all $\bu \in \R^K$ such that $\|\bu\| = 1$, $\bu^\top \Big( \bC_n^{(i,j)} - \gamma_\C \Diag(\bC_n^{(i,j)}) \Big) \bu \leq (1 - \frac{\gamma_\C}{K}) \bu^\top \bC_n^{(i,j)} \bu \leq (1 - \frac{\gamma_\C}{K})  \|\bC_n^{(i,j)}\|_{\rm S}$. Combining this result with \eqref{majoration C_n i+1} and \eqref{majoration C_n i bar}, we deduce that $\|\bC_n^{(i,j+1)}\|_{\rm S} \leq (1 - \frac{\gamma_\C}{K})  \|\bC_n^{(i,j)}\|_{\rm S} + \gamma_\C + \sqrt{\frac{c_\C}{2}}$, then we can prove by recurrence that for all $(i,j)$, $\|\bC_n^{(i,j)}\|_{\rm S} \leq \max \big( \|\bC_n^{(0)}\|_{\rm S},K(1 + \sqrt{\frac{1}{2 \alpha \gamma_\C}}) \big)$. Consequently, if we define 
\begin{equation}
\label{rho_bar}
\bar{\rho} = \max \big( \|\rho_{C^{(0)}}\|_{\rm S},K(1 + \sqrt{\frac{1}{2 \alpha \gamma_\C}}) \big),
\end{equation}

 $\bar{\rho}$ is an upper bound for $(\rho_{\C^{(i)}})_{i \in \N}$, which proves that $(\C^{(i)})_{i \in \N}$ is bounded.\\

$\bullet$ {Boundedness of \texorpdfstring{$(\W^{(i)})_{i \in \N}$}{Wi}}\\
Again, we fix $i \in \N$, the proximal operator used in Algorithm \ref{PALM-IVA-G algorithm} ensures that $\C^{(i)} \in (\epsilon \bI_K + \S^+_K)^N$. Thus, $(\forall n \in \nfirst) \, \tr(\bC^{(i)}_n \bW^{(i)}_n \Rx \bW^{(i)\top}_n) \geq \epsilon \tr (\bW^{(i)}_n \Rx \bW^{(i)\top}_n) \geq \epsilon \sigma^- \|\bW^{(i)}_n\|^2$ using \eqref{WRW>lambda||W||}. By summing these inequalities, 
\begin{equation}
\label{minoration h(W,C)}
    \frac{1}{2}\sum_{n=1}^N \tr(\bC^{(i)}_n \bW^{(i)}_n \Rx \bW^{(i)\top}_n) \geq \frac{\epsilon \sigma^-}{2} \|\W^{(i)}\|^2.
\end{equation}
Besides, let us consider the concave inequality on the logarithm function at point $\frac{2}{\epsilon \sigma^-} > 0$,
 \begin{equation}
 \label{logconcave}
 (\forall x \in (0,+\infty))\quad \log x \leq \log \frac{2}{\epsilon \sigma^-} + \frac{\epsilon \sigma^-}{2} x - 1
. \end{equation}

Taking $k \in \kfirst$ and summing \eqref{logconcave} applied  on $\|\bw_n^{(i)[k]}\|$ for $n \in \nfirst$, we obtain
\begin{equation}
    \sum_{n=1}^N \log \|\bw_n^{(i)[k]}\| \leq N (\log \frac{2}{\epsilon \sigma^-} -1) + \frac{\epsilon \sigma^-}{2} \|\bW^{(i)[k]}\|^2.
\end{equation}
And combining this inequality with \eqref{minoration term 1} gives
\begin{equation*}
    -\log |\det \bW^{(i)[k]}| \geq - \frac{\epsilon \sigma^-}{4} \|\bW^{(i)[k]}\|^2  - \frac{N}{2} (\log \frac{2}{\epsilon \sigma^-} -1).
\end{equation*}
Finally, summing for $k \in \kfirst$,
\begin{equation}
\label{minoration f(W)}
    - \sum_{k=1}^K \log |\det \bW^{(i)[k]}| \geq - \frac{\epsilon \sigma^-}{4} \|\W^{(i)}\|^2 - \frac{KN}{2} (\log \frac{2}{\epsilon \sigma^-} -1).
\end{equation}
Summing \eqref{minoration f(W)} and \eqref{minoration h(W,C)} yields
\begin{align*}
    &\frac{\epsilon \sigma^-}{4} \|\W^{(i)}\|^2 - \frac{KN}{2} (\log \frac{2}{\epsilon \sigma^-} -1)\\
    &\leq \frac{1}{2} \sum_{n=1}^N \tr(\bC^{(i)}_n \bW^{(i)}_n \Rx \bW^{(i)\top}_n) - \sum_{k=1}^K \log |\det \bW^{(i)[k]}|\\
    &= J_{\rm IVA-G}(\W^{(i)},\C^{(i)}) + \frac{1}{2} \sum_{n=1}^N \log \det \bC^{(i)}_n\\
    &\leq J_{\rm IVA-G}(\W^{(i)},\C^{(i)}) + \frac{KN}{2} \log \bar{\rho}.
\end{align*}

Following the proof of \cite[Lemma 2]{bolte2014proximal}, it is sufficient that the first four items of Assumption \ref{as:Palm} hold to obtain that $J_{\rm IVA-G}^{\text{Reg}}(\W^{(i)},\C^{(i)})$ is a non-increasing sequence. We thus have

\begin{align}
     \|\W^{(i)}\|^2 \leq \frac{4}{\epsilon \sigma^-} \Big(J_{\rm IVA-G}^{\text{Reg}}(\W^{(0)},\C^{(0)}) + \frac{KN}{2} \log \bar{\rho} (\frac{2}{\epsilon \sigma^-} -1) \Big).
\end{align}

This shows that $(\W^{(i)})_{i \in \N}$ is bounded too.



\paragraph{Bounds for the corrected Lipschitz moduli}

Using ~\eqref{eq:lipW}, we can set $\lambda_\W^- = \epsilon \varrho_{\Rx}$ and $\lambda_\W^+ = \bar{\rho} \varrho_{\Rx}$. We can also set $\lambda_\C^- = \lambda_\C^+ = \alpha$.

\paragraph{KL property}

KL property is a key tool from functional analysis to demonstrate convergence of iterates in the non-convex setting~\cite{bolte2014proximal}. It is sufficient here to apply the result from \cite[Section 4.3]{PAM_Proj}, which states that a function verifies KL on the domain of its subdifferential, if it is lower semi-continuous, proper, and definable in an o-minimal structure. 

We rely on the o-minimal structure $\mathfrak{S}(\R_{\rm an,exp})$, defined in \cite[Section 2, Example (6)]{van_den_dries_geometric_1996}, which contains $\log, \exp$ and all semi-algebraic functions. The properties in \cite[Section 5]{van_den_dries_geometric_1996} can then be used to show that $\mathfrak{S}(\R_{\rm an,exp})$ contains our cost function $J_{\rm IVA-G}^{\text{Reg}}$, after identifying $\R^{N \times N \times K} \times \R^{K \times K \times N}$ with $\R^{KN(K+N)}$.

\end{appendices}

\section*{Acknowledgment}
The work of E.C. and C.C. is funded by the European Research Council Starting Grant MAJORIS ERC-2019-STG-850925.

\def\url#1{}
\bibliographystyle{IEEEtran}
\bibliography{IEEEabrv,Bibliography}

\end{document}